\documentclass{amsart}

\usepackage{graphicx}
\usepackage{morefloats,bm,caption,amsbsy,enumerate,amsmath,amsthm,amssymb,mathtools,amsfonts,multirow,verbatim,tikz,tikz-cd,thmtools,thm-restate}
\usepackage{enumitem}
\usepackage[alphabetic]{amsrefs}
\usepackage{chngcntr}

\usepackage[hidelinks,colorlinks=true,linkcolor=blue, citecolor=black,linktocpage=true]{hyperref}
\usetikzlibrary{matrix,arrows,decorations.pathmorphing}
\usepackage[top=1.25in, bottom=1in, left=1.25in, right=1.25in,marginpar=1in]{geometry}
\usepackage{url}
\counterwithin{figure}{section}

\newtheorem{innercustomthm}{Theorem}
\newenvironment{customthm}[1]
  {\renewcommand\theinnercustomthm{#1}\innercustomthm}
  {\endinnercustomthm}

\newtheorem{thm}{Theorem}[section]
\newtheorem{prop}[thm]{Proposition}

\newtheorem{cor}[thm]{Corollary}
\newtheorem{lem}[thm]{Lemma}

\theoremstyle{definition}
\newtheorem{define}[thm]{Definition}

\theoremstyle{remark}
\newtheorem{rem}[thm]{Remark}
\newtheorem{example}[thm]{Example}

\newcommand{\ve}[1]{\boldsymbol{\mathbf{#1}}}
\newcommand{\R}{\mathbb{R}}

\newcommand{\Z}{\mathbb{Z}}
\newcommand{\N}{\mathbb{N}}

\newcommand{\C}{\mathbb{C}}

\newcommand{\T}{\mathbb{T}}

\renewcommand{\d}{\partial}
\renewcommand{\subset}{\subseteq}

\renewcommand{\tilde}{\widetilde}
\renewcommand{\bar}{\overline}

\newcommand{\iso}{\cong}

\newcommand{\ufrs}{\underline{\frs}}

\DeclareMathOperator{\cl}{{cl}}

\DeclareMathOperator{\GL}{{GL}}
\DeclareMathOperator{\gr}{{gr}}

\DeclareMathOperator{\id}{{id}}

\DeclareMathOperator{\Int}{{int}}

\DeclareMathOperator{\Spin}{{Spin}}

\DeclareMathOperator{\Sym}{{Sym}}
\DeclareMathOperator{\Symp}{{Symp}}

\DeclareMathOperator{\Lef}{{Lef}}

\newcommand{\bF}{\mathbb{F}}

\newcommand{\bK}{\mathbb{K}}
\newcommand{\bL}{\mathbb{L}}

\newcommand{\bT}{\mathbb{T}}
\newcommand{\bU}{\mathbb{U}}

\newcommand{\CP}{\mathbb{CP}}

\newcommand{\cA}{\mathcal{A}}

\newcommand{\cD}{\mathcal{D}}

\newcommand{\cH}{\mathcal{H}}

\newcommand{\cL}{\mathcal{L}}
\newcommand{\cM}{\mathcal{M}}

\newcommand{\cO}{\mathcal{O}}

\newcommand{\cS}{\mathcal{S}}
\newcommand{\cT}{\mathcal{T}}

\newcommand{\cW}{\mathcal{W}}

\newcommand{\frT}{\mathfrak{T}}

\newcommand{\frs}{\mathfrak{s}}

\newcommand{\CF}{\mathit{CF}}
\newcommand{\CFu}{\underline{\mathit{CF}}}

\newcommand{\HFKh}{\widehat{\mathit{HFK}}}
\newcommand{\HFLh}{\widehat{\mathit{HFL}}}

\newcommand{\CFLh}{\widehat{\mathit{CFL}}}
\newcommand{\SFH}{\mathit{SFH}}

\newcommand{\HFK}{\mathit{HFK}}

\newcommand{\CFL}{\mathit{CFL}}
\newcommand{\HFL}{\mathit{HFL}}

\newcommand{\PD}{\mathit{PD}}

\newcommand{\xs}{\ve{x}}
\newcommand{\ys}{\ve{y}}
\newcommand{\zs}{\ve{z}}
\newcommand{\ws}{\ve{w}}

\newcommand{\ps}{\ve{p}}

\newcommand{\as}{\ve{\alpha}}
\newcommand{\bs}{\ve{\beta}}

\newcommand{\taus}{\ve{\tau}}

\newcommand{\omegas}{\ve{\omega}}

\renewcommand{\a}{\alpha}
\renewcommand{\b}{\beta}
\newcommand{\g}{\gamma}

\renewcommand{\GL}{\mathit{GL}}

\usepackage{leftidx}

\DeclareMathOperator{\Irr}{Irr}

\DeclareMathOperator{\tw}{tw}

\newcommand{\tra}{\mathfrak{T}}
\newcommand{\page}{F}

\title{Transverse invariants and exotic surfaces in the 4-ball}

\author{Andr\'as Juh\'asz}
\address{Mathematical Institute, University of Oxford, Andrew Wiles Building,
Radcliffe Observatory Quarter, Woodstock Road, Oxford, OX2 6GG, UK}
\email{juhasza@maths.ox.ac.uk}

\author{Maggie Miller}
\address{Department of Mathematics\\Princeton University\\  Princeton, NJ 08544, USA}
\email{maggiem@math.princeton.edu}

\author{Ian Zemke}
\address{Department of Mathematics\\Princeton University\\  Princeton, NJ 08544, USA}
\email{izemke@math.princeton.edu}

\begin{document}


\begin{abstract}
Using 1-twist rim surgery, we construct infinitely many smoothly embedded, orientable surfaces in the 4-ball
bounding a knot in the 3-sphere that are pairwise topologically isotopic, but not ambient diffeomorphic.
We distinguish the surfaces using the maps they induce on perturbed link Floer homology.
Along the way, we show that the cobordism map induced by an ascending surface in a Weinstein
cobordism preserves the transverse invariant in link Floer homology.
\end{abstract}

\maketitle


\section{Introduction}

Let $S$ be a smooth surface in the smooth 4-manifold $X$. Then we say that
the surface $S'$ is an \emph{exotic} copy of $S$ if $S$ and $S'$ are
topologically isotopic, but not smoothly isotopic. It is a fundamental open
conjecture in the theory of knotted surfaces that there are no exotic
orientable unknots of any genus in $S^4$.
In fact, currently there is no example of any exotic pair of closed oriented surfaces in $S^4$.
Note that, in the non-orientable case, Finashin, Kreck, and Viro~\cite{unoriented}
constructed an infinite family of exotic copies of the standardly embedded
$\#_{10} \R P^2$ in $S^4$. We provide the first examples of exotic orientable
surfaces in the 4-ball, and in fact distinguish them up to diffeomorphism
(that is not necessarily the identity on $S^3$)
using perturbed cobordism maps on link Floer homology.
We note that Gompf~\cite[Theorem~8.1]{Gompf-nuclei} constructed punctured tori in $B^4$
that are non-diffeomorphic, and he conjectured to be pairwise homeomorphic.

\begin{thm}\label{thm:main}
  There are infinitely many knots in $S^3$ such that each bounds countably infinitely
  many properly embedded, smooth, orientable, genus one surfaces in $B^4$ that are
  pairwise topologically isotopic, but there is no diffeomorphism of $B^4$
  taking one to the other.
\end{thm}

Exotic pairs of orientable surfaces in other 4-manifolds are constructed using variations of the rim
surgery operation of Fintushel and Stern~\cite{rim}. For example, Finashin~\cite{Finashin} constructed
smoothly inequivalent surfaces in $\C P^2$,
which Kim and Ruberman proved to be topologically isotopic~\cite{Kim-Ruberman2}.
To show topological isotopy, previous constructions require the surface complement to be
simply-connected, or at least to have finite cyclic fundamental group; see
Kim and Ruberman~\cite{Kim-Ruberman2}\cite{Kim-Ruberman}. The surfaces
are distinguished using Seiberg--Witten invariants.
It is unclear if these methods give rise to exotic surfaces in $S^4$.

In Theorem~\ref{thm:topiso}, we provide a method for constructing surfaces
that are topologically isotopic (and are hence potentially exotic), where the
fundamental group of the surface complement is irrelevant. We use twist rim
surgery, introduced by Kim~\cite{Kim}. This combines the twist-spinning
construction of 2-knots, due to Zeeman~\cite{twisting}, with rim surgery.
Zeeman showed that a 1-twist-spun 2-knot is always smoothly trivial. This can
be rephrased as follows: 1-twist rim surgery on an unknotted 2-sphere in
$S^4$ is smoothly trivial. Building on this result, in
Section~\ref{sec:isotopy}, we show that, if $S$ is a surface in a 4-manifold,
and $\gamma \subset S$ is a simple closed curve that bounds a topologically
embedded disk $D$ in the complement of $S$, then the result $S'$ of 1-twist
rim surgery on $S$ along $\gamma$ is topologically isotopic to $S$. When $D$
is smooth, then $S'$ is smoothly isotopic to $S$. In fact, our result holds
for any \emph{concordance rim surgery} -- a generalization of twist rim
surgery that we introduce in Section~\ref{sec:rim} --
that gives the unknotted 2-sphere when performed
along the equator of $S^2 \subset S^4$. Note that Kim and
Ruberman~\cite[Corollary~4.6]{Kim-Ruberman2} gave a different proof
of Theorem~\ref{thm:topiso} in the case of 1-twist rim surgery that relies on
the 4-dimensional topological $s$-cobordism theorem, assuming the
fundamental group of the complement of $S$ is \emph{good} in the sense of Freedman,
and that the knot used for the rim surgery is slice.

Hence, if we find a pair $(S, \gamma)$, where $\gamma$ bounds a topological
disk, but not a smooth disk, in the complement of $S$, then we can construct
infinitely many potentially exotic copies of $S$. This is the case whenever
$S$ is a Seifert surface in $S^3$ pushed into $B^4$, and $\g \subset S$
comes from a non-separating, Alexander polynomial one knot on the Seifert surface
with trivial surface framing. We perform 1-twist rim surgery on this pair $(S,\gamma)$.
(Note that $\pi_1(B^4 \setminus S) \cong \Z$, since $S$ comes from $S^3$.)

What remains is to show that the resulting surfaces are not diffeomorphic.
For this, we use the cobordism maps induced by the surfaces on
perturbed link Floer homology. The effect of concordance rim surgery on these
maps follows from the work of the first and third
authors~\cite{JZConcordanceSurgery}; see Theorem~\ref{thm:concordance-rim-surgery}
for the precise formula.

To distinguish the maps, we need that
the map induced by $S$ is non-vanishing, that $\g$ is homologically
non-trivial on $S$, and that the pattern we use for the 1-twist rim surgery
has non-trivial Alexander polynomial; see Theorem~\ref{thm:not-diffeomorphic}.
We achieve the first condition by
finding a quasipositive $S$, and showing that such surfaces induce
non-vanishing maps on link Floer homology,
as they preserve the transverse invariant in knot Floer homology
defined by Lisca, Ozsv\'ath, Stipsicz, and Szab\'o~\cite{LOSS}, and extended by
Baldwin, Vela-Vick, and V\'ertesi~\cite{BVVVTransverse}
(usually referred to as the LOSS or BRAID invariant); see Corollary~\ref{cor:BRAID}.
For example, $S$ can be the standard genus one Seifert surface of a twice iterated,
positive, untwisted Whitehead double of any nontrivial, strongly quasipositive knot.

In fact, Corollary~\ref{cor:BRAID} is a special case of the following much more general result:

\begin{thm}\label{thm:ascending-Stein}
Suppose $(W,\cS) \colon (Y_0,\bL_0) \to (Y_1,\bL_1)$ is a decorated link
cobordism such that $W$ has a Weinstein structure $(W,\omega,\phi,V)$,
and $\cS=(S,\cA)$, where $S$ is an ascending surface with positive critical
points. If the decoration $\cA$ is $\ws$-anti-arboreal with respect to
$\phi$, then
\[
\left(F^\circ_{W,\cS}\right)^{\vee}\left(\tra^\circ(\bL_1)\right) = \tra^\circ(\bL_0)
\]
for $\circ\in \{\wedge,-\}$, where $\tra^\circ(\bL_i)$ is the transverse invariant
of the transverse link $\bL_i$ for $i \in \{0, 1\}$.
\end{thm}

See Section~\ref{sec:constructible} for some background on ascending surfaces in
Weinstein cobordisms, and Definition~\ref{def:arboreal-decoration} for $\ws$-anti-arboreal
decorations. We prove Theorem~\ref{thm:ascending-Stein} in Section~\ref{sec:functoriality}.
Theorem~\ref{thm:ascending-Stein} fits into the context of a family of similar results due to
Ozsv\'ath and Szab\'o~\cite{OSContactStructures}*{Theorem~1.5},
the first author~\cite{JCob}*{Theorem~11.24}, Baldwin and Sivek~\cite{BaldwinSivekEquivalence}*{Theorem~1.10} \cite{BaldwinSivek-LegendrianMonopole}*{Theorem~1.2}, Golla and the first author~\cite{Golla-Juhasz},
Baldwin, Lidman, and Wong~\cite{BLWLegendrian}*{Theorem~1.5}, and Kang~\cite{Kang-Naturality}\cite{Kang-Transverse}.

To distinguish the surfaces up to diffeomorphism using the knot Floer cobordism maps,
we define an invariant $\Omega(S) \in \Z^{\ge 0} \cup \{-\infty\}$ for any orientable surface of genus $g > 0$
bounding a knot $K$ in $S^3$. We prove that $\Omega(S)$ is an invariant of the diffeomorphism type
of the pair $(B^4,S)$ in Theorem~\ref{thm:diffeomorphism-invariant}. Furthermore, if $S$ is a surface bounding $K$,
and $S'$ is obtained by 1-twist rim surgery using an auxiliary knot $J$ in $S^3$, we prove that
\[
\Omega(S') = \Omega(S) + \Irr(\Delta_J(t)),
\]
where $\Irr(\Delta_J(t))$ is the number of irreducible factors of the Alexander polynomial of $J$.
Finally, if $S$ is a quasipositive surface in $B^4$, we prove that
\[
\Omega(S) = 0;
\]
see Proposition~\ref{prop:ascending-perturbed-coefficients}.
In fact, an algebraic argument of Sunukjian~\cite[Theorem~2]{Sunukjian-surgery}, 
combined with Theorem~\ref{thm:ascending-Stein}, implies that $S$ and $S'$
are never diffeomorphic whenever $S$ is a quasipositive surface and
$\Delta_J(t) \neq 1$, though this is not immediately sufficient for Theorem~\ref{thm:main}.

Instead of the link cobordism maps, one could distinguish the surfaces
we construct up to diffeomorphism \emph{fixing $S^3$ pointwise} using Seiberg--Witten theory,
as follows: By Rudolph~\cite{Rudolph-algebraic}, every quasipositive surface $S$ in $B^4$
is algebraic. Hence, we can consider the projectivization $\hat{S}$ of $S$ in $\CP^2$, and perturb it to be
nonsingular. Then one can apply the result of Kim~\cite[Theorem~3.4]{Kim-Hee-Jung}
to compute the effect of twist rim surgery on the relative Seiberg--Witten invariant $\text{SW}_{\CP^2, \hat{S}}$
of the complement of $\hat{S}$. Note that this method does not extend to general concordance
surgeries, due to the lack of a concordance surgery formula on the Seiberg--Witten side.

It is a natural question whether one can use Theorem~\ref{thm:main} to obtain
an exotic pair of orientable \emph{closed} surfaces in $S^4$.
A natural candidate would be obtained by taking the surfaces $S$ and $S'$,
constructed in Theorem~\ref{thm:main}, and considering $S \cup -S'\subset S^4$.
However, it is a straightforward exercise to see that this surface is smoothly unknotted.

\subsection{Acknowledgements}

We would like to thank Peter Feller for his invaluable guidance
on quasipositive surfaces. We would also like to thank
Anthony Conway, Robert Gompf, Kyle Hayden, Hee Jun Kim, Anubhav Mukherjee, Mark Powell,
and Daniel Ruberman for helpful discussions.

The first author was supported by a Royal Society Research Fellowship,
the second author by an NSF Graduate Research Fellowship (DGE-1656466),
and the third author by an NSF Postdoctoral Research Fellowship (DMS-1703685).
This project has received funding from the European Research Council (ERC)
under the European Union's Horizon 2020 research and innovation programme
(grant agreement No 674978).

\section{Concordance rim surgery and topological isotopy}

\subsection{Concordance rim surgery}
\label{sec:rim}

Concordance surgery is a generalization of knot surgery introduced by
Fintushel and Stern; see Akbulut~\cite{AkbulutGokova} and the work
of the first and third authors~\cite{JZConcordanceSurgery}.
In this section, we present a generalization of rim surgery,
based on concordance surgery.

Suppose $S$ is a properly embedded, smooth, oriented surface in the smooth, compact 4-manifold $W$,
and $\g$ is a simple closed curve in $S$. We identify neighborhood of $\g$ in $(W,S)$
with $(S^1 \times B^3, S^1 \times a)$, where $a \subset B^3$ denotes
a properly embedded, unknotted arc.

\begin{define}\label{def:rim1}
Suppose $S$ is a properly embedded, smooth, oriented surface in the compact 4-manifold $W$,
and $\g \subset S$ is a simple closed curve. Furthermore, let $(I \times S^3, C)$ be a
self-concordance of a knot $K \subset S^3$. We assume that there is a neighborhood
$N(p) \subset S^3$ of a point $p \in K$ such that the arc $N(p) \cap K$ is unknotted, and
\[
(I \times N(p)) \cap C = I \times (N(p) \cap K).
\]

If we glue $\{0\} \times S^3$ and $\{1\} \times S^3$ using the equivalence
relation $(0,x) \sim (1,x)$ for $x \in S^3$, then we obtain the annulus
$A := (C \setminus (I \times N(p)))/{\sim}$ in
$S^1 \times B^3 \approx I \times (S^3 \setminus N(p))/{\sim}$.
Let $N(\g)$ be a regular neighborhood of $\g$ in $W$ such that $S \cap N(\g)$
is an annulus, and choose an orientation-preserving diffeomorphism
\[
d \colon (S^1 \times S^2, \d A) \to (\d N(\g), S \cap \d N(\g)),
\]
the isotopy class of which is determined by a normal framing of $S$ along $\g$.

The \emph{concordance rim surgery} $S(\g, C)$ of $S$ along $\g$ with pattern $C$ is obtained
by gluing $(S^1 \times B^3, A)$ to $(W \setminus N(\g), S \setminus N(\g))$ using the
diffeomorphism $d$ (which we suppress in our notation).
For an illustration, see Figure~\ref{fig:concordance-surgery}.
\end{define}

\begin{figure}[ht!]
\centering
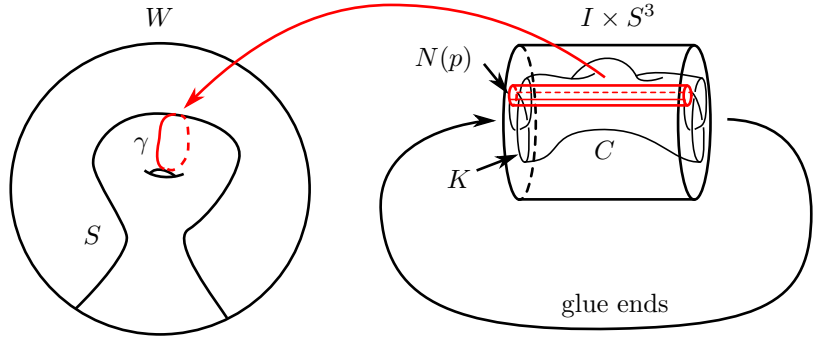
\caption{Concordance rim surgery on a surface $S$ along the curve $\gamma$ with pattern
a self-concordance $C$ of a knot $K$.}
\label{fig:concordance-surgery}
\end{figure}

\begin{rem}\label{rem:rim}
  Alternatively, one could define concordance rim surgery as follows:
  Let $S$ be a properly embedded, smooth, oriented surface in the smooth 4-manifold $W$,
  together with a simple closed
  curve $\gamma \subset S$. We denote by $T$ the \emph{rim torus} around $\gamma$, which
  is the union of the fibers of the unit normal circle bundle of $S$ in $X$ over $\gamma$.
  Furthermore, let $(I \times S^3, C)$ be a self-concordance of a knot $K$ in $S^3$.
  Then we could define $S(\gamma, C)$
  as the surface obtained by doing concordance surgery along the rim torus $T$
  using the self-concordance $C$.

  Kim~\cite[Lemma~2.4]{Kim} showed the equivalence of the two definitions
  in the special case of twist rim surgery, which we review in
  the next section. We do not use this equivalence when proving our results,
  so we do not study the relationship between the two possible definitions of concordance
  rim surgery in general.
\end{rem}

\subsection{Topological isotopy}
\label{sec:isotopy}

In this section, we show that 1-twist rim surgery along a curve
that bounds a topological disk in the surface complement preserves the
topological isotopy type of the surface. We first review the definition
of 1-twist rim surgery, due to Kim~\cite{Kim}.

\begin{define}\label{def:1-twist}
Let $K$ be an oriented knot in $S^3$, and let $T = \partial N(K)$ be
the boundary of a regular neighborhood $N(K)$ of $K$.
Fix an identification $T \approx S^1 \times S^1$ such that $S^1 \times \{\varphi\}$
is a meridian of $K$ for $\varphi \in S^1$. Let $V \approx T \times I$ be a regular neighborhood of $T$
in $S^3 \setminus K$ such that $T \times \{0\} \subset N(K)$.
Furthermore, choose a smooth monotonic function $f \colon \R \to I$
such that $f(t) = 0$ for $t \le 0$ and $f(t) = 1$ for $t \ge 1$.
We define an automorphism $\phi$ of $(S^3, K)$ by
\[
\phi(x) =
\begin{cases}
x & \text{if } x \not\in V, \\
(\theta + 2\pi f(t), \varphi, t) & \text{if } x = (\theta, \varphi, t) \in V \approx S^1 \times S^1 \times I.
\end{cases}
\]
For $n \in \Z$, the \emph{$n$-twist self-concordance} of $K$ is given by the tuple
\[
(I \times S^3, I \times K, h_0, h_1),
\]
where $h_i \colon \{i\} \times S^3 \to S^3$ are given by $h_0(0, x) = x$
and $h_1(1, x) = \phi^n(x)$.
We denote the 1-twist self-concordance of $K$ by $(I \times S^3, K_{\tw})$.

Let $S$ be a properly embedded, smooth, oriented surface in the compact 4-manifold $W$,
and $\g \subset S$ a simple closed curve, together with a normal framing of $S$ along $\g$.
Then we call $S(\g, K_{\tw})$ the \emph{1-twist rim surgery} of $S$ along $\g$ with pattern $K$.
\end{define}

When $K$ is nontrivial, the knot cobordisms $(I \times S^3, K_{\tw})$ and $(I \times S^3, I \times K)$
are inequivalent, which follows from the work of Zeeman~\cite{twisting}.

\begin{rem}\label{rem:zeeman}
The 1-twist self-concordance first appeared in the work of Zeeman~\cite{twisting}
in slightly different language. If $S^2$ is the standard 2-sphere in $S^4$ with equator $S^1$,
and $K$ is a knot in $S^3$, then $S^2(S^1, K_{\tw})$ is the 1-twist-spun 2-knot obtained from $K$
in the terminology of Zeeman. Here, we use the normal framing of $S^2$ along $S^1$
induced by $D^2$.

In particular, Zeeman showed that, if $S$ is a 2-sphere embedded in the 4-manifold $W$
that bounds a locally flat, topologically (resp.~smoothly) embedded ball,
and $K$ is a knot in $S^3$, then the 2-sphere $S(\gamma, K_{\tw})$ also bounds
a locally flat, topologically (resp.~smoothly) embedded ball in $W$.
\end{rem}

\begin{thm}\label{thm:topiso}
Let $S$ be a properly embedded, oriented surface in a 4-manifold $W$.
Suppose $\gamma \subset S$ a simple closed curve that bounds a locally flat,
topologically embedded disk $D$ in $W$, such that $\Int(D) \cap S = \emptyset$.
Furthermore, let $C$ be a self-concordance of a knot in $S^3$ such that
the concordance rim surgered 2-sphere $S^2(S^1, C)$ is topologically unknotted in $S^4$.
Then $S(\gamma, C)$, the result of concordance rim surgery on $S$ along $\gamma$ with pattern $C$
and normal framing of $S$ along $\g$ given by $D$, is topologically isotopic to $S$.
\end{thm}

\begin{proof}
We understand everything in this proof to take place in the topological,
locally flat category. We use the existence and uniqueness of normal bundles up
to ambient isotopy, and the existence of handle
decompositions in this category, for which we refer the reader to Chapter~9
of the book of Freedman and Quinn~\cite{FQ}.

Let $D \times I$ be a thickening of $D$, such that $\partial D \times I \subset S$.
Let $S'$ be the surface obtained by smoothing the corners of
$(S \setminus (\d D \times I)) \cup (D \times \d I)$.
Choose a closed 4-ball $B$ containing $D \times I$ with
$S' \subset W \setminus \Int(B)$ and $S' \cap \partial B =
D \times \d I$; see the left and middle of Figure~\ref{fig:topiso}.

Note $S$ is obtained from $S'$ by surgery along a 3-dimensional 1-handle with
core arc $\eta = \{c\} \times I$, where $c$ is the center point of the disk $D$.
The framing on $\gamma$ is
determined by specifying that the result $S$ of this surgery is orientable.

Let $U$ be the (smoothed) symmetric difference of $S(\gamma, C)$ and $S'$,
and push $U$ off of $S'$ and into $B$. By construction, the pair $(B, U)$
is homeomorphic to $(2B^4, S^2(S^1, C))$, where $2 B^4 \subset \R^4$ is a ball of radius~2. By
assumption, this implies that $U$ is an unknotted 2-sphere bounding a 3-ball
$\Delta$ in $B$. Now $S(\gamma, C)$ is obtained from $S' \cup U$ by surgery
along two 3-dimensional 1-handles with core arcs $\eta_1$, $\eta_2 \subset B$,
where $\eta_1 = \{c\} \times [0,1/4]$ and $\eta_2 = \{c\} \times [3/4,1]$. There are two
choices of framings for $\eta_1$; the choice will determine the framing for
$\eta_2$; see the right of Figure~\ref{fig:topiso}.

\begin{figure}[ht!]
\centering
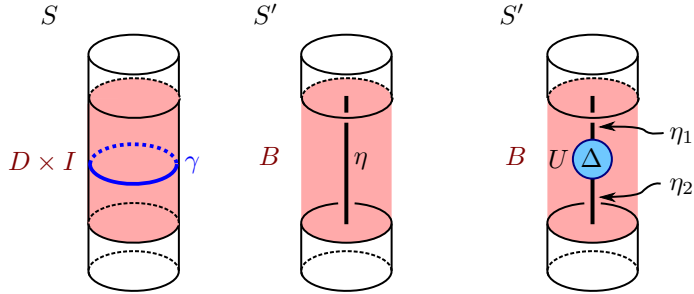
\caption{Left: A thickening of the locally flat disk $D$. The surface $S'$ is
obtained from $S$ by compressing along $D$. Middle: We obtain $S$ from $S'$
by surgery along a 3-dimensional 1-handle with core $\eta$. Right: We obtain
$S(\gamma,C)$ from $S' \sqcup U$ by surgery along two 3-dimensional 1-handles
with cores $\eta_1$, $\eta_2$. The sphere $U$ is the symmetric difference of
$S(\gamma,C)$ and $S'$, pushed off of $S'$. Then $U$ is the result of
concordance surgery on an unknotted sphere with pattern $C$,
which is unknotted, by assumption.}
\label{fig:topiso}
\end{figure}

Since $\pi_1(B \setminus \Delta) \cong 1$, we can isotope each $\eta_i$ rel
boundary in $B$ to be disjoint from $\Delta$. By concatenating the
3-dimensional 1-handles about $\eta_1$ and $\eta_2$ with $\Delta$, we find that
$S(\gamma,C)$ is obtained from $S'$ by surgery along a 3-dimensional 1-handle
with core arc $\eta' \subset B$ (and with framing determined by the
orientability of $S(\gamma,C)$). Since $\eta$ and $\eta'$ are both contained
in the ball $B$, we conclude that $\eta$ and $\eta'$ are isotopic rel
boundary as framed arcs in $W \setminus S'$, and therefore that $S$ and
$S(\gamma,C)$ are isotopic.
\end{proof}

\begin{rem}
In Theorem~\ref{thm:topiso}, if we take $D$ to be smoothly embedded, then by the same argument in the
smooth category we find that $S(\gamma,C)$ is smoothly isotopic to $S$.

When the disk $D$ is topologically but not smoothly embedded, we cannot
conclude that $S(\gamma, C)$ is \emph{smoothly} obtained from $S'$ by
surgery along a 3-dimensional 1-handle. Thus, the argument fails smoothly,
and we do not necessarily expect that $S(\gamma, C)$ and $S$ are smoothly isotopic
in general.
\end{rem}

The most well-known situation in which concordance surgery on an unknotted sphere yields
the unknot is the case of 1-twist rim surgery \cite{twisting}.

\begin{cor}
Let $S$ be a properly embedded, oriented surface in a 4-manifold $W$, and $\gamma \subset S$
a simple closed curve. Suppose $\gamma$ bounds a locally flat, topologically
embedded disk $D$ in $W$, such that $\Int(D) \cap S = \emptyset$. Then $S(\gamma, K_{\tw})$
is topologically isotopic to $S$ for any knot $K$ in $S^3$,
where the normal framing of $S$ along $\g$ is given by $D$.	
\end{cor}

There are other concordance surgeries on the unknotted 2-sphere that
yield the unknot; e.g., $n$-twist 1-roll rim surgery with pattern $K$,
assuming $S^3_n(K)$ is a lens space; see Teragaito~\cite{Teragaito}
and Litherland \cite{Litherland}. For our
main examples, we will primarily consider 1-twist rim surgery.

\begin{example}
  Let $S$ be a quasipositive Seifert surface of a knot $J$ in $S^3$,
  and $\g$ a homologically nontrivial simple closed curve on $S$ with
  trivial Alexander polynomial and trivial surface framing.
  Then $\g$ bounds a topological disk $D$ in $B^4$ that only intersects $S$
  in $\g$. If we perform surgery along $D$, we obtain a topological surface
  $S'$ bounding $J$ of smaller genus. Quasipositivity of $S$ implies it minimizes the slice genus.
  Hence, the topological slice genus of $J$ is
  less than its smooth slice genus, and the curve $\g$ does not bound
  a smooth disk in the complement of $S$. Compare Rudolph's counterexample to the locally flat Thom conjecture \cite{RudolphSlice}*{Section~5}.

  In particular, we can apply Theorem~\ref{thm:topiso} to the pair $(S, \g)$
  to see that the 1-twist rim surgery $S(\g, K_{\tw})$ is topologically isotopic to $S$,
  for any knot $K$ in $S^3$.
  Since $\g$ does not bound a smooth disk in the complement of $S$, it is possible
  that $S$ and $S(\g, K_{\tw})$ are not smoothly isotopic. In the rest of the paper,
  we will prove that this is indeed the case whenever $\Delta_K(t) \neq 1$,
  using perturbed cobordism maps on link Floer homology.
\end{example}

\section{Link Floer homology}
\label{sec:backgroundHFK}

In this section, we review some background on the knot and link Floer
homology groups of Ozsv\'{a}th and Szab\'{o} \cite{OSKnots} \cite{OSLinks},
as well as the constructions of link cobordism maps due to the first and third
authors \cite{JCob} \cite{ZemCFLTQFT}.

\subsection{The link Floer homology groups}

If $(\Sigma,\as,\bs,\ws,\zs)$ is a Heegaard diagram for a multi-pointed link $(Y,\bL)$,
then there are link Floer homology groups
\begin{equation}
\HFLh(\Sigma,\as,\bs,\ws,\zs) \quad\text{and} \quad \HFL^-(\Sigma,\as,\bs,\ws,\zs).
\end{equation}
These are homologies of chain complexes $\CFLh(\Sigma,\as,\bs,\ws,\zs)$ and
$\CFL^-(\Sigma,\as,\bs,\ws,\zs)$, respectively, defined as follows.

Let $\bF$ be the field of two elements, and $\bF[v]$ the polynomial ring in the variable $v$.
We consider the two  half-dimensional tori
\[
\bT_\a:=\a_1\times \cdots \times \a_n\quad \text{and} \quad \bT_{\b}:=\b_1\times \cdots \times \b_n
\]
in $\Sym^n(\Sigma)$, where $n=g(\Sigma)+|\ws|-1$.
Let $\CFLh(\Sigma,\as,\bs,\ws,\zs)$ be the $\bF$-vector space generated by
intersection points $\xs \in \bT_{\a}\cap \bT_{\b}$. Furthermore, let
$\CFL^-(\Sigma,\as,\bs,\ws,\zs)$ be the free $\bF[v]$-module generated by
intersection points $\xs \in \bT_{\a}\cap \bT_{\b}$.

The differential on $\CFLh(\Sigma,\as,\bs,\ws,\zs)$ is given by the formula
\[
\d \xs = \sum_{\ys\in \bT_{\a}\cap \bT_{\b}} \sum_{\substack{\phi\in \pi_2(\xs,\ys)\\
\mu(\phi) = 1\\ n_{\ws}(\phi)=n_{\zs}(\phi)=0}} \# \left(\cM(\phi)/\R\right) \cdot \ys,
\]
where $\cM(\phi)$ denotes the moduli space of pseudo-holomorphic disks in $\Sym^n(\Sigma)$
representing the class $\phi$.
The differential on $\CFL^-(\Sigma,\as,\bs,\ws,\zs)$ is given by the formula
\[
\d \xs = \sum_{\ys\in \bT_{\a}\cap \bT_{\b}} \sum_{\substack{\phi\in \pi_2(\xs,\ys)\\
\mu(\phi) = 1 \\ n_{\ws}(\phi)=0}} \# \left(\cM(\phi)/\R \right) v^{n_{\zs}(\phi)}\cdot \ys,
\]
extended equivariantly over the action of $v$.

The chain complexes $\CFLh(\Sigma,\as,\bs,\ws,\zs)$ and
$\CFL^-(\Sigma,\as,\bs,\ws,\zs)$ have refinements over $\Spin^c$ structures.
For $\frs\in \Spin^c(Y)$, we define $\CFLh(\Sigma,\as,\bs,\ws,\zs, \frs)$
as the submodule generated by intersection
points $\xs$ which satisfy $\frs_{\ws}(\xs) = \frs$, and similarly for $\CFL^-$.
When $c_1(\frs)$ is torsion, and the link $\bL$ is null-homologous, there are two
Maslov gradings, which we denote by $\gr_{\ws}$ and $\gr_{\zs}$. With respect to $(\gr_{\ws}, \gr_{\zs})$,
the differential has bigrading $(-1,-1)$, and the action of $v$ has bigrading $(0,-2)$.

\subsection{Link cobordism maps}

In this section, we summarize the functorial properties of link Floer homology,
due to the first and third authors \cite{JCob} \cite{ZemCFLTQFT}.

\begin{define}
 Suppose that $(Y_0,\bL_0)$ and $(Y_1,\bL_1)$ are two multi-pointed links.
 Write $\bL_i=(L_i,\ws_i,\zs_i)$ for $i \in \{0, 1\}$. A \emph{decorated link cobordism} from $(Y_0,\bL_0)$
 to $(Y_1,\bL_1)$ consists of a pair $(W,\cS)$, satisfying the following:
\begin{enumerate}
\item $W$ is a compact, oriented 4-manifold with $\d W=-Y_0\cup Y_1$.
\item $\cS=(S,\cA)$, where $S$ is a compact, oriented, and properly embedded surface in $W$, such that $\d S=-L_0\cup L_1$.
\item $\cA$ is a set of properly embedded arcs on $S$,
    whose endpoints are disjoint from $\ws_i$ and $ \zs_i$ for $i\in \{0,1\}$,
    and such that each component of $L_i\setminus (\ws_i\cup \zs_i)$
    contains exactly one endpoint of $\d \cA$. Furthermore,
    $S\setminus \cA$ consists of two subsurfaces, denoted $\cS_{\ws}$ and $\cS_{\zs}$, that meet along $\cA$, such that $\ws_i\subset \cS_{\ws}$ and $\zs_i\subset \cS_{\zs}$.
\item Each component of $W$ contains a component of $S$, and
    each component of $S$ contains a component of $\cA$.
\end{enumerate}
\end{define}

For a decorated link cobordism $(W,\cS)\colon (Y_0,\bL_0)\to (Y_1,\bL_1)$,
the first author \cite{JCob} constructed a cobordism map
\[
\hat{F}_{W,\cS}\colon \HFLh(Y_0,\bL_0)\to \HFLh(Y_1,\bL_1).
\]
The third author \cite{ZemCFLTQFT} subsequently constructed $\bF[v]$-equivariant maps
\[
F^-_{W,\cS}\colon \HFL^-(Y_0,\bL_0)\to \HFL^-(Y_1,\bL_1).
\]
The construction in \cite{ZemCFLTQFT} also induces a cobordism map on $\HFLh$, which is of a different flavor than the one in \cite{JCob}.
The first and third authors \cite{JuhaszZemkeContactHandles}*{Theorem~1.4}
proved that the two constructions give the same maps on $\HFLh$.

\subsection{Duality and link Floer homology}

We now recall some basic results about duality and link Floer homology, which
feature prominently in the functorial properties of the transverse
invariants. Firstly, if $\cH=(\Sigma,\as,\bs,\ws,\zs)$ is a Heegaard link
diagram for $(Y,\bL)$, then $\cH^\vee=(\Sigma,\bs,\as,\ws,\zs)$ is a
diagram for $(-Y,-\bL)$, where $-\bL=(-L,\ws,\zs)$ denotes
$\bL$, with orientation reversed.

There is a canonical isomorphism
\[
\CFLh(\Sigma,\bs,\as,\ws,\zs)\iso \left(\CFLh(\Sigma,\as,\bs,\ws,\zs)\right)^\vee,
\]
where $\vee$ denotes duality of $\bF$-vector spaces; see
\cite{OSKnots}*{Proposition~3.7}. Similarly, the chain complexes $\CFL^-(\Sigma,\bs,\as,\ws,\zs)$
and $\CFL^-(\Sigma,\as,\bs,\ws,\zs)^\vee$ are canonically isomorphic, where
$\vee$ now denotes the dual as a chain complex over $\bF[v]$.

The link cobordism maps also satisfy an analogous duality property.
If $(W,\cS)\colon (Y_0,\bL_0)\to (Y_1,\bL_1)$ is a decorated link cobordism,
then turning around $(W,\cS)$ gives a link cobordism
\[
(W,\cS)^\vee:=(W^\vee, \cS^\vee)\colon (-Y_1,-\bL_1)\to (-Y_0,-\bL_0).
\]
Furthermore, using the description of the link cobordism maps in terms
of elementary cobordisms from \cite{ZemCFLTQFT},
it is straightforward to adapt \cite{OSTriangles}*{Theorem~3.5} to obtain that
\begin{equation}
F^\circ_{(W,\cS)^\vee}=(F_{W,\cS}^\circ)^{\vee}, \label{eq:duality-maps}
\end{equation}
for $\circ\in \{\wedge,-\}$.

\section{Perturbed sutured Floer homology}
\label{sec:totally-twisted}

\emph{Sutured Floer homology} is an invariant of sutured manifolds, due to the first author~\cite{JDisks}.
\emph{Perturbed sutured Floer homology}, introduced by the first and third authors~\cite{JZConcordanceSurgery}, is a refinement for sutured manifolds equipped with a collection of closed
2-forms.

Sutured Floer homology perturbed by $n$ closed 2-forms
has coefficients in the group ring $\bF[\R^n]$, which we note contains $\bF[\N^n]$ and $\bF[\Z^n]$.
These are the rings of polynomials and Laurent polynomials, respectively. If $(a_1,\dots, a_n) \in \R^n$,
we write $ e^{(a_1,\dots, a_n)} $ for the
corresponding element of the group ring, which we think of as the monomial
$z_1^{a_1}\cdots z_n^{a_n}$.

If $\omegas=(\omega_1,\dots, \omega_n)$ is a tuple of closed 2-forms on a manifold $X$,
then there is an action of the group $C_2(X;\Z)$ of smooth 2-chains on $\bF[\R^n]$, given by
\begin{equation}
e^{h} \cdot e^{(a_1,\dots, a_n)} = e^{(a_1+\int_h \omega_1,\dots, a_n+\int_h \omega_n)}.
\label{eq:module-action-perturbed}
\end{equation}
We write $\bF[\R^n]_{\omegas}$ for $\bF[\R^n]$ equipped with this action.

\begin{define}\label{def:proj-int}
If $x \in \bF[\R^{n}]$, then we say that $x$ is \emph{projectively integral} if
\[
x = m \cdot y
\]
for some $y \in \bF[\Z^n]$ and monomial $m = e^{(a_1,\dots, a_n)} \in \bF[\R^n]$.
More generally, if $M_0$ is an $\bF$-vector space, and
\[
M := M_0 \otimes \bF[\R^n],
\]
we say that $x \in M$ is \emph{projectively integral} if there is some monomial
$m \in \bF[\R^n]$ such that $m \cdot x \in M_0 \otimes \bF[\Z^n]$.
\end{define}

\begin{define}
If $X$ is a smooth manifold, we say that a closed 2-form $\omega\in \Omega^2(X)$
is \emph{integral} if $\int_F \omega \in \Z$ for any closed,
singular 2-chain $F\in C_2(X;\Z)$.
\end{define}

If $\omegas$ is an  $n$-tuple of closed 2-forms on a sutured manifold
$(M,\g)$, the first and third authors \cite{JZConcordanceSurgery} described a perturbed version of
sutured Floer homology, denoted
$\SFH(M,\g;\bF[\R^n]_{\omegas})$. If $\cW$ is a cobordism between the sutured
manifolds $(M_0,\g_0)$ and $(M_1,\g_1)$, and $\omegas$ is a collection of
closed 2-forms on $\cW$, then the first and third authors also constructed a
perturbed version of the cobordism map
\[
F_{\cW;\omegas}\colon \SFH(M_0,\g_0;\bF[\R^n]_{\omegas_0})\to \SFH(M_1,\g_1;\bF[\R^n]_{\omegas_1}),
\]
where $\omegas_i=\omegas|_{M_i}$. A sketch of the construction may be found later in this section.

The main technical result of this section is the following:

\begin{prop}\label{prop:integrality}
Suppose that $S \subset B^4$ is a properly embedded, oriented surface
intersecting $S^3$ in a knot $K$, and $\omegas = (\omega_1,\dots,
\omega_{n})$ is a collection of closed, integral 2-forms on $B^4 \setminus N(S)$ that vanish on $S^3
\setminus N(K)$. Then, for any dividing set $\cA$ on $S$, the element
\[
F_{B^4, \cS; \omegas}(1) \in \HFKh(S^3,K) \otimes \bF[\R^n]
\]
is projectively integral,
where $\cS=(S,\cA)$ and $F_{B^4, \cS; \omegas}$ is the cobordism map on perturbed sutured
Floer homology~\cite{JZConcordanceSurgery} over $\bF[\R^n]$ induced by the sutured manifold
cobordism complementary to $\cS$ from the empty sutured manifold to $S^3(K)$,
where $\SFH(S^3(K)) \cong \HFKh(S^3,K)$.
\end{prop}

\subsection{A totally twisted version of sutured Floer homology}

To show Proposition~\ref{prop:integrality}, it is convenient to lift the map on
perturbed sutured Floer homology to a totally twisted map, in the spirit of the version described by
Ozsv\'{a}th and Szab\'{o} for ordinary Heegaard Floer homology \cite{OSTriangles}*{Section~2.7}.

If $X$ is a topological space,
write $C_k(X)$ for the group of smooth, integral, singular $k$-chains in $X$.
Let $B_2(X)$ be the image of $\d \colon C_3(X) \to C_2(X)$,
and write $(C_2/B_2)(X)$ for the quotient $C_2(X)/B_2(X)$.

Let $(M,\g)$ be a balanced sutured manifold with admissible diagram $(\Sigma, \as, \bs)$,
and let $\ufrs$ be a relative $\Spin^c$ structure on $(M,\g)$.
If $L$ is a $\bF[(C_2/B_2)(M)]$-module, we define the twisted sutured Floer complex
\[
\CFu(\Sigma, \as, \bs,\ufrs; L)
\]
to be the group generated by elements $\xs \otimes h$, where $\xs \in \T_{\a} \cap \T_{\b}$
satisfies $\frs(\xs)=\ufrs$, and $h \in L$. The differential is given by
\[
\d(\xs \otimes h) = \sum_{\ys \in \bT_{\a} \cap \T_{\b}} \sum_{\substack{
\phi \in \pi_2(\xs,\ys) \\ \mu(\phi) = 1}} \# (\cM(\phi)/\R) \cdot \ys \otimes h\cdot e^{\tilde{\cD}(\phi)},
\]
where $\tilde{\cD}(\phi) \in C_2(M)$ is obtained by coning off the domain of $\phi$
using sets of compressing disks for $\as$ and $\bs$, which are implicit in the construction,
as in \cite[Section~2.2]{JZConcordanceSurgery}.
If $\omega$ is an $n$-tuple of closed 2-form on $M$, then the perturbed complex
$\CF(\Sigma, \as, \bs,\ufrs;\bF[\R]_{\omegas})$ is the tensor product
\[
\CFu(\Sigma, \as, \bs,\ufrs; \bF[(C_2/B_2)(M)]) \otimes_{\bF[(C_2/B_2)(M)]} \bF[\R^n]_{\omegas},
\]
where $\bF[(C_2/B_2)(M)]$ acts on $\bF[\R^n]$ as in equation~\eqref{eq:module-action-perturbed}.

We write
\[
\CFu(\Sigma, \as, \bs; L) = \bigoplus_{\ufrs \in \Spin^c(M,\g)} \CFu(\Sigma, \as, \bs,\ufrs; L).
\]
There are inclusion and projection maps
\[
i_{\ufrs} \colon \CFu(\cH, \ufrs; L) \to \CFu(\cH; L) \quad \text{and} \quad
\pi_{\ufrs} \colon \CFu(\cH; L) \to \CFu(\cH,\ufrs; L),
\]
where $\cH = (\Sigma, \as, \bs)$.

The chain complexes $\CFu(\Sigma, \as, \bs,\ufrs; L)$ for a given sutured manifold $(M,\g)$
and $\ufrs \in \Spin^c(M,\g)$ form a projective transitive system that we denote by $\CFu(M, \g, \ufrs; L)$.
I.e., the only monodromy of the transition maps is multiplication by $e^h$ for $h \in (C_2/B_2)(M)$,
up to chain homotopy; see \cite[Section~2.1]{JZConcordanceSurgery}.
The transition maps for changing diagrams are defined using straightforward extensions
of the formulas below for cobordism maps; also see \cite[Section~6]{JZConcordanceSurgery}.
In general, summing over $\Spin^c$ structures will not produce a natural invariant,
as in \cite{JZConcordanceSurgery}*{Section~6.5}.

Let $\cW = (W, Z, [\xi])$ be a balanced sutured manifold cobordism from $(M_0, \g_0)$ to $(M_1, \g_1)$.
We can view $\bF[(C_2/B_2)(W)]$ as a $\bF[(C_2/B_2)(M_k)]$-module
using the embedding $i_k \colon M_k \hookrightarrow W$ for $k \in \{0, 1\}$.
We now describe a totally twisted map
\[
\underline{F}_{\cW} \colon \CFu(M_0,\g_0; \bF[(C_2/B_2)(M_0)]) \to
\CFu(M_1,\g_1; \bF[(C_2/B_2)(W)]),
\]
as follows.
We decompose $\cW$ into a boundary cobordism $\cW^\d$ and a
special cobordism $\cW^{s}$, as in \cite{JCob}. We further decompose $\cW^s$ into
cylinders, and 1-handle, 2-handle, and 3-handle cobordisms.
As in \cite[Proposition~2.9]{JZConcordanceSurgery}, only the map
\[
\pi_{\ufrs_1} \circ \underline{F}_{\cW} \circ i_{\ufrs_0} \colon
\CFu(M_0,\g_0, \ufrs_0; \bF[(C_2/B_2)(M_0)]) \to
\CFu(M_1,\g_1, \ufrs_1; \bF[(C_2/B_2)(W)])
\]
is well-defined, up to multiplication by $e^h$ for $h \in (C_2/B_2)(W)$, and chain homotopy.

We begin by defining the totally twisted map for cylinders. If $W$ has a
Morse function $f$ with no critical points
and gradient-like vector field $v$, then the flow of $v/v(f)$ gives a diffeomorphism
between $W$ and $I \times M_0$. If $\cH_0$ is an admissible diagram for $M_0$, we
obtain an admissible diagram $\cH_1$ for $(M_1,\g_1)$ by using the flow of $v/v(f)$.
If $\xs$ is an intersection point for $\cH_0$, there is a
corresponding intersection point of $\cH_1$, for which we write $v_*(\xs)$.
We let $\Gamma_{\xs}$ be the singular 2-chain obtained by sweeping out
$\gamma_{\xs}$ under the flow of $v/v(f)$. The twisted map for $\cW$ in this
case is
\[
\underline{F}_{\cW}(\xs \otimes e^h)=v_*(\xs)\otimes e^{(i_0)_*(h) + \Gamma_{\xs}},
\]
where $h \in (C_2/B_2)(M_0)$, and $(i_0)_* \colon (C_2/B_2)(M_0) \to (C_2/B_2)(W)$
is induced by the embedding $i_0 \colon M_0 \hookrightarrow W$.
The twisted cobordism maps for 1-handles and 3-handles are defined using
a similar formula, so we leave the details to the reader. Compare \cite{JZConcordanceSurgery}*{Section~7.3}.

We now focus on 2-handle cobordisms, and follow \cite[Section~7.4]{JZConcordanceSurgery}.
We pick a Morse function $f$ on $W$ with
gradient-like vector field $v$ that is Morse--Smale and has only index 2 critical
points. Let $(\Sigma,\as,\bs,\bs')$ be a Heegaard triple subordinate to a
bouquet for the framed link in $M_0$ induced by $(f,v)$. We obtain an embedding of $W_{\a,\b,\b'}$ into $W$,
which is well-defined up to isotopy. If $\psi \in \pi_2(\xs, \Theta_{\b,\b'},\ys)$
is a class of triangles with $\xs \in \bT_\a \cap \bT_\b$ and $\ys \in \bT_\a \cap \bT_{\b'}$,
we obtain a 2-chain $\tilde{\cD}(\psi)$ in $W$ by coning off the domain $\cD(\psi)$.
The 2-handle map is defined via the formula
\[
\underline{F}_{\cW}(\xs \otimes e^h) = \sum_{\substack{\ys \in \bT_\a \cap \bT_{\b'}\\ \frs(\ys) = \ufrs_1}}
\sum_{\substack{\psi \in \pi_2(\xs,\Theta_{\b,\b'},\ys)\\ \mu(\psi) = 0}}
\# \cM(\psi) \cdot \ys \otimes e^{(i_0)_*(h) + \tilde{\cD}(\psi)},
\]
where $\xs \in \bT_\a \cap \bT_\b$ satisfies $\frs(\xs) = \ufrs_0$, and $h \in (C_2/B_2)(M_0)$.

The contact gluing map extends to this setting, as follows. Let $(M,\g)$ be a
sutured submanifold of $(M',\g')$, and $\xi$ a positive contact structure
on $M' \setminus \Int(M)$ that induces the dividing set $\g \cup \g'$.
For $\ufrs \in \Spin^c(M, \g)$ represented by the nowhere vanishing vector field $v$ on $M$,
we obtain $\ufrs' \in \Spin^c(M',\g')$ by gluing $v$ to $\xi^\perp$.
Then the contact gluing map
\[
\underline{\Phi}_{\xi} \colon \CFu(-M,\g, \ufrs ;\bF[(C_2/B_2)(M)]) \to \CFu(-M',\g', \ufrs' ;\bF[(C_2/B_2)(M')])
\]
is defined via the formula
\[
\underline{\Phi}_{\xi}(\xs \otimes e^h) = \Phi_{\xi}(\xs)\otimes e^{i_*(h)},
\]
where $\frs(\xs) = \ufrs'$, the map $i_* \colon (C_2/B_2)(M) \to (C_2/B_2)(M')$
is induced by the inclusion of $M$ into $M'$, and $\Phi_{\xi}$ is the untwisted gluing map.

Finally, as in \cite[Section~7.5]{JZConcordanceSurgery},
we define the totally twisted cobordism map $\underline{F}_{\cW}$
for a general balanced cobordism $\cW$ by composing the totally twisted
contact gluing map to obtain $\underline{F}_{\cW^\d}$,
followed by the totally twisted maps for the cylinders and handle cobordisms
to obtain $\underline{F}_{\cW^s}$.

One may follow the proof of invariance of the perturbed cobordism maps from
\cite[Section~7.5]{JZConcordanceSurgery} to see that the $\Spin^c$ restricted totally twisted cobordism map
$\pi_{\ufrs_1} \circ \underline{F}_{\cW} \circ i_{\ufrs_0}$ is well-defined,
up to overall multiplication by $e^h$ for $h \in (C_2/B_2)(W)$, and chain homotopy.

By construction, if $\omegas=(\omega_1,\dots, \omega_n)$ is a collection of closed 2-forms on $W$,
then one obtains the perturbed map $F_{\cW;\omegas}$
by tensoring with the group ring $\bF[\R^n]_{\omegas}$,
which is a module over $(C_2/B_2)(W)$ with the action shown in
equation~\eqref{eq:module-action-perturbed}; i.e.,
\begin{equation}\label{eq:pair-with-twisted}
F_{\cW; \omegas} = \underline{F}_{\cW} \otimes 1_{\bF[\R^n]_{\omegas}}.
\end{equation}

\subsection{Proof of Proposition~\ref{prop:integrality}}

We now proceed with the main details of the proof of Proposition~\ref{prop:integrality}. We begin with a lemma:

\begin{lem} \label{lem:integrality-totally-twisted}
Suppose that $\cW = (W,Z,[\xi]) \colon (M_0,\g_0) \to (M_1,\g_1)$ is a balanced sutured
manifold cobordism, and $\cH_0$ and $\cH_1$ are admissible diagrams for $(M_0,\g_0)$ and
$(M_1,\g_1)$, respectively.  Suppose $\xs_0$ and $\xs_0'$ are intersection
points on $\cH_0$, and $\xs_1$ and $\xs_1'$ are intersection points on $\cH_1$, such that
\[
\frs(\xs_0) = \frs(\xs_0') \in \Spin^c(M_0,\g_0) \quad \text{and} \quad
\frs(\xs_1) = \frs(\xs_1') \in \Spin^c(M_1,\g_1).
\]
Suppose $\xs_1 \otimes e^{h_{\xs_0}}$ appears as a summand of
$\underline{F}_{\cW}(\xs_0)$ and $\xs_1' \otimes e^{h_{\xs_0'}}$ appears as a
summand of $\underline{F}_{\cW}(\xs_0')$, for some $h_{\xs_0}$, $h_{\xs_0'} \in (C_2/B_2)(W)$.
Let $\phi_0 \in \pi_2(\xs_0,\xs_0')$ and $\phi_1\in \pi_2(\xs_1,\xs_1')$. Then
\[
h_{\xs_0} - h_{\xs_0'} + (i_1)_* \tilde{\cD}(\phi_1)-(i_0)_* \tilde{\cD}(\phi_0)
\]
is a closed 2-chain in $W$, where $(i_0)_*$ and $(i_1)_*$
are the maps induced by the inclusions of $M_0$ and $M_1$ into $W$.
\end{lem}

\begin{proof}
It is sufficient to show the claim separately for the twisted contact gluing map, the
1-handle, 2-handle, and 3-handle maps, as well as the transition maps for
changing the Heegaard diagram.

The claim for the contact gluing maps, the 1-handle maps, and the 3-handle
maps are straightforward, so we focus on the 2-handle maps. The
claim for the transition maps for changing the Heegaard diagram is an easy
modification of the argument we present for 2-handles. Suppose $(f,v)$ is a
Morse--Smale pair on $W$, and $f$ has only index 2 critical points.
Suppose $(\Sigma,\as,\bs,\bs')$ is a Heegaard triple subordinate to a bouquet for the
framed link in $M_0$ induced by $(f,v)$. We obtain an embedding of the
3-ended cobordism $W_{\a,\b,\b'}$ into $W$. Suppose
$\psi \in \pi_2(\xs_0,\Theta_{\b,\b'},\xs_1)$ and
$\psi'\in \pi_2(\xs_0',\Theta_{\b,\b'},\xs_1')$. Let $\phi_0 \in \pi_2(\xs_0,\xs_0')$ and
$\phi_1 \in \pi_2(\xs_1,\xs_1')$ be classes of disks, as in the statement.

We simply note that
\begin{equation}
\begin{split}
\d \tilde{\cD}(\phi_0) = \gamma_{\xs_0'} - \gamma_{\xs_0},
&\quad \d \tilde{\cD}(\phi_1) = \gamma_{\xs_1'}-\gamma_{\xs_1},\\
\d \tilde{\cD}(\psi) = C_{\a,\b,\b'} + \gamma_{\xs_1} - \gamma_{\Theta_{\b,\b'}} - \gamma_{\xs_0},
\quad &\text{and} \quad  \d\tilde{\cD}(\psi') =
C_{\a,\b,\b'} + \gamma_{\xs_1'} - \gamma_{\Theta_{\b,\b'}} - \gamma_{\xs_0'},
\end{split}
\label{eq:compute-boundary-integrality}
\end{equation}
where $C_{\a,\b,\b'}$ is the 1-chain in $W$, defined as follows. The 4-manifold
$W_{\a,\b,\b'}$ is constructed by gluing $\Sigma \times \Delta$, $U_{\a} \times e_a$,
$U_{\b} \times e_{\b}$, and $U_{\b'} \times e_{\b'}$, where $\Delta$
denotes a triangle, $U_{\a}$, $U_{\b}$, and $U_{\b'}$ are standard sutured
compression bodies, and $e_{\a}$, $e_{\b}$, and $e_{\b'}$ are intervals that
are identified with the sides of the triangle $\Delta$. Given a collection of
compressing disks for $U_{\a}$, we write $c_{\a}$ for their center points, and
we define the 1-chain $C_{\a} = e_{\a} \times c_{\a}$. We define the 1-chains
$C_{\b}$ and $C_{\b'}$ similarly. The 1-chain $C_{\a,\b,\b'}$ is the sum of
$C_{\a}$, $C_{\b}$, and $C_{\b'}$ (in particular, it is independent of the
choice of triangle or intersection point).
By equation~\eqref{eq:compute-boundary-integrality}, we have
\[
\d\tilde{\cD}(\psi) - \d\tilde{\cD}(\psi') + \d (i_1)_* \tilde{\cD}(\phi_1) - \d (i_0)_* \tilde{\cD}(\phi_0) = 0,
\]
and the result follows.
\end{proof}

\begin{proof}[Proof of Proposition~\ref{prop:integrality}]
Write $\cW$ for the sutured manifold cobordism complementary to $S$, viewed
as a cobordism from from $\emptyset$ to $S^3(K)$, the sutured manifold complementary to $K$.

The map $F_{\cW;\omegas}$ satisfies an Alexander grading formula;
see \cite{JMComputeCobordismMaps}*{p.~3} and
\cite{ZemAbsoluteGradings}*{Theorem~1.4}. In particular, all summands of
$F_{\cW;\omegas}(1)$ reside in the same Alexander grading, which is
equivalent to representing the same relative $\Spin^c$ structure on $S^3(K)$.

We apply Lemma~\ref{lem:integrality-totally-twisted} to see that, if
$\xs \otimes e^{h} $ and $ \xs' \otimes e^{h'}$ are summands of the totally
twisted map $\underline{F}_{\cW}(1)$ for some $h$, $h' \in (C_2/B_2)(W)$,
and $\phi \in \pi_2(\xs,\xs')$ is a class of disks, then
\[
h - h' + \tilde{\cD}(\phi)
\]
is a closed 2-chain. We obtain that
\[
\int_{h - h' + \tilde{\cD}(\phi)} \omega_i \in \Z
\]
for every $i \in \{1,\dots,n\}$, since $\omega_i$ is integral by assumption.
Note that $\omega_i$ also vanishes on $S^3 \setminus N(K)$ by assumption, so
$\int_{\tilde{\cD}(\phi)} \omega_i = 0$, and hence
\[
\int_h \omega_i - \int_{h'} \omega_{i} \in \Z.
\]
Since the perturbed map is obtained from the totally twisted one by tensoring
with $\bF[\R^n]_{\omegas}$, we conclude that, after multiplying
$F_{\cW;\omegas}(1)$ by some $e^{(a_1,\dots, a_n)}$, we obtain an element of
$\HFKh(K)\otimes \bF[\Z^n]$, concluding the proof.
\end{proof}

\section{Concordance rim surgery and knot Floer homology}

In this section, we compute the effect of concordance rim surgery on the
perturbed cobordism maps. Suppose $S$ is a properly
embedded, oriented surface in $B^4$ intersecting $S^3$ in a
knot $K$. We identify a neighborhood $N(S)$ of $S$ with $S \times D^2$. It is
straightforward to investigate the Mayer--Vietoris sequence for the
decomposition of $B^4$ as the union of $B^4 \setminus N(S)$ and $N(S)$ to see that
\begin{equation*}
H^2(B^4 \setminus N(S), S^3 \setminus N(K)) \iso H^1(S,\d S).
\end{equation*}
Geometrically, this isomorphism can be described by
taking a properly embedded surface
in $B^4 \setminus N(S)$ with boundary in $\d N(S)$,
taking its intersection with $S \times S^1$, and
projecting to $S$. Hence, we can identify a basis of
$H^2(B^4 \setminus N(S), S^3 \setminus N(K))$ with a basis of $H_1(S)$.

If $C$ is a self-concordance of a knot $J$ in $S^3$,
then it is also straightforward to see that
\begin{equation}\label{eq:H^2-complement-concordance}
H^2((S^1 \times S^3) \setminus N(C)) \iso \Z.
\end{equation}
A generator is given by the Poincar\'e dual of $\{0\} \times F \subset
\{0\} \times (S^3 \setminus J)$, where $F$ is a Seifert surface of $J$.

Note also that there is a canonical isomorphism
\begin{equation}\label{eq:canonical-iso}
H^2(B^4 \setminus N(S), S^3 \setminus N(K)) \iso
H^2(B^4 \setminus N(S(\g,C)), S^3\setminus N(K)).
\end{equation}

The goal of this section is to prove the following:

\begin{thm}\label{thm:concordance-rim-surgery}
Suppose that $S \subset B^4$ is a properly embedded surface with boundary
$K \subset S^3$, and let $\omegas=(\omega_1,\dots,\omega_n)$ denote an $n$-tuple of closed 2-forms
in $H^2(B^4 \setminus N(S), S^3 \setminus N(K);\R)$.  Let $\gamma \subset S$ be a
simple closed curve, and let $T_\gamma \subset \d N(S)$ denote the 2-torus
which is the preimage of $\gamma$ with respect to the projection $\d N(S) \to S$.
If $C$ is a self-concordance of a knot $J$, and $S(\g,C)$ is the surface
obtained by concordance rim surgery on $S$ along $\g$ with pattern $C$, then
\[
F_{B^4, S(\g,C);\omegas}(1) = \Lef_{Z(\g,\omegas)}(C) \cdot F_{B^4,S;\omegas}(1),
\]
where
\[
Z(\g,\omegas) = z_1^{\int_ {T_\g} \omega_1} \cdots z_{n}^{\int_{T_\g} \omega_{n}}.
\]
Here, we are giving $S(\g,C)$ and $S$ the decoration where the $\ws$-subregion is a bigon,
and the $\zs$-subregion is a genus $g(S)$ subsurface.
\end{thm}

We will prove Theorem~\ref{thm:concordance-rim-surgery} in Section~\ref{sec:proof-rim-formula}.

\subsection{The Floer homology of the 2-fiber link in \texorpdfstring{$S^1\times S^2$}{S1 X S2}}

In this section, we describe the Floer homology of a simple link in $S^1 \times S^2$
that appears when we do rim surgery. Suppose $S$ is a properly embedded, oriented
surface in $B^4$, and $\g \subset S$ is a simple closed curve.
Given a normal framing of $S$ along $\g$, we may identify a neighborhood of $\g$ in $(B^4, S)$
with $(S^1 \times B^3, S^1 \times a)$, where $a \subset B^3$ denotes
a properly embedded, unknotted arc.
We may view $(S^1 \times B^3, S^1 \times a)$ as a link cobordism from $\emptyset$ to
\[
(S^1 \times S^2, L_2),
\]
where $L_2$ consists of two $S^1$-fibers of $S^1 \times S^2$, with opposite orientations.
Let $\bL_2$ be $L_2$ decorated with four basepoints, two on each component.
A Heegaard diagram for $(S^1 \times S^2, \bL_2)$ is shown in Figure~\ref{fig:6}.

\begin{figure}[ht!]
	\centering
	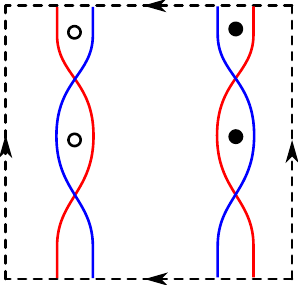
	\caption{A Heegaard diagram for $(S^1 \times S^2, \bL_2)$.}
    \label{fig:6}
\end{figure}

\begin{lem}\label{lem:compute-HFK-L2}
The vector space $\HFLh(S^1 \times S^2,\bL_2)$ has rank 4.
As a $(\gr_{\ws},\gr_{\zs})$-bigraded vector space, we have
\[
\HFLh(S^1 \times S^2,\bL_2)=(\bF)_{(1,-1)} \oplus (\bF^2)_{(0,0)} \oplus (\bF)_{(-1,1)}.
\]
Furthermore, if $\omega$ is any closed 2-form on $(S^1 \times S^2) \setminus N(L_2)$, then
\[
\HFLh(S^1 \times S^2, \bL_2, i; \bF[\R]_{\omega}) \iso
\HFLh(S^1 \times S^2, \bL_2, i) \otimes \bF[\R]
\]
for each Alexander grading $i \in \Z$.
\end{lem}

\begin{proof}
Using the diagram in Figure~\ref{fig:6}, we have
\[
\CFLh(S^1 \times S^2, \bL_2) = (\bF)_{(1,-1)} \oplus
(\bF^2)_{(0,0)} \oplus (\bF)_{(-1,1)}.
\]
The differential vanishes, since there are no index 1 classes
that have zero multiplicity at the four basepoints.
\end{proof}

\begin{rem}
A basis for $\HFLh(S^1 \times S^2, \bL_2)$ may also be specified using the
images of the cobordism maps for $(S^1 \times B^3, S^1 \times a)$,
with the four dividing sets shown in Figure~\ref{fig:5}. This may be proven using the composition law,
and standard TQFT-style arguments, though we will not need this fact.
\end{rem}

\begin{figure}[ht!]
\centering
\begingroup%
  \makeatletter%
  \providecommand\color[2][]{%
    \errmessage{(Inkscape) Color is used for the text in Inkscape, but the package 'color.sty' is not loaded}%
    \renewcommand\color[2][]{}%
  }%
  \providecommand\transparent[1]{%
    \errmessage{(Inkscape) Transparency is used (non-zero) for the text in Inkscape, but the package 'transparent.sty' is not loaded}%
    \renewcommand\transparent[1]{}%
  }%
  \providecommand\rotatebox[2]{#2}%
  \newcommand*\fsize{\dimexpr\f@size pt\relax}%
  \newcommand*\lineheight[1]{\fontsize{\fsize}{#1\fsize}\selectfont}%
  \ifx\svgwidth\undefined%
    \setlength{\unitlength}{255.4789148bp}%
    \ifx\svgscale\undefined%
      \relax%
    \else%
      \setlength{\unitlength}{\unitlength * \real{\svgscale}}%
    \fi%
  \else%
    \setlength{\unitlength}{\svgwidth}%
  \fi%
  \global\let\svgwidth\undefined%
  \global\let\svgscale\undefined%
  \makeatother%
  \begin{picture}(1,0.32937527)%
    \lineheight{1}%
    \setlength\tabcolsep{0pt}%
    \put(0,0){\includegraphics[width=\unitlength,page=1]{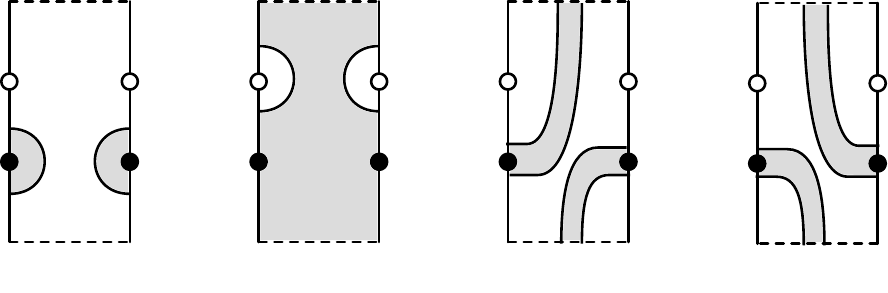}}%
    \put(0.07835441,0.00551966){\color[rgb]{0,0,0}\makebox(0,0)[t]{\lineheight{1.25}\smash{\begin{tabular}[t]{c}$\cA$\end{tabular}}}}%
    \put(0.35939535,0.00599366){\color[rgb]{0,0,0}\makebox(0,0)[t]{\lineheight{1.25}\smash{\begin{tabular}[t]{c}$\bar{\cA}$\end{tabular}}}}%
  \end{picture}%
\endgroup%

\caption{Dividing sets on an annulus.  The link cobordism maps for
$(S^1 \times B^3, S^1 \times a)$ with these dividing sets give a basis of
$\HFLh(S^1 \times S^2,\bL_2)$.}
\label{fig:5}
\end{figure}

\subsection{Proof of the concordance rim surgery formula}
\label{sec:proof-rim-formula}

\begin{proof}[Proof of Theorem~\ref{thm:concordance-rim-surgery}]
We factor both link cobordisms through a regular neighborhood of the curve $\g$.
Consider the annuli
\[
A := S \cap N(\g) \quad \text{and} \quad A' := S(\g,C) \cap N(\g).
\]
Let $\cA$ and $\cA'$ denote decorations of $A$ and $A'$, respectively,
such that the $\ws$-subregion consists of two bigons, as
on the left side of Figure~\ref{fig:5}.
Identify $N(\g)$ with $S^1 \times B^3$. We let $\cW$ and $\cW'$ denote the
two sutured manifold cobordisms that are complementary to the decorated link cobordisms
\[
(S^1 \times B^3, \cA) \quad \text{and} \quad (S^1 \times B^3, \cA'),
\]
respectively.

Write $\taus = \omegas|_{N(\g)\setminus N(S)}$ and
$\taus' = \omegas_{N(\g)\setminus N(S(\g,C))}$, where
$\taus = (\tau_1, \dots, \tau_{2g})$ and $\taus' = (\tau_1',\dots,\tau_{2g}')$.
Using the isomorphism from equation~\eqref{eq:H^2-complement-concordance},
we may assume that
\begin{equation}
\tau_i = \alpha_i \cdot \PD[\{0\} \times F_U] \quad \text{and} \quad \omega_i' =
\alpha_i \cdot \PD[\{0\} \times F_K],\label{eq:classes-PD}
\end{equation}
for some $\alpha_i \in \R$, where $F_U$ and $F_K$ are Seifert surfaces
for the unknot $U$ and for $K$, respectively.

Using the composition law, it suffices to show that
\begin{equation}\label{eq:local-rim-surgery-formula}
F_{\cW';\taus'} \doteq \Lef_{Z(\g,\omegas)}(C) \cdot F_{\cW;\taus},
\end{equation}
where $\doteq$ denotes equality up to multiplication by a monomial.
Both maps have $(\gr_{\ws},\gr_{\zs})$-bigrading $(1,-1)$, and hence we can view both maps as
having the same rank 1 codomain over $\bF[\R^{2g}]$ by
Lemma~\ref{lem:compute-HFK-L2}.

We now consider the link cobordism
$(S^1\times B^3, \bar{A})$, obtained by reversing and turning around the
orientation of $\cW$, which we view as a link cobordism from $(S^1\times S^2, L_2)$
to the empty set. Write $\bar{\cA}$ for $\bar{A}$ decorated with the
second dividing set from the left in Figure~\ref{fig:5}. Write $\bar{\cW}$
for the sutured manifold cobordism complementary to $\bar{\cA}$ in $S^1
\times B^3$, viewed as a cobordism from $(S^1\times S^2,\bL_2)$ to the empty
set. We let $\bar{\taus}$ denote the tuple of 2-forms on $(S^1\times
B^3)\setminus N(\bar{A})$ which are scalar multiples of the Poincar\'{e}
dual of $-F_U$, as in equation~\eqref{eq:classes-PD}.

Since we are viewing the domain and codomain of both $F_{\cW';\taus'}$ and $F_{\cW;\taus}$ as rank 1 over $\bF[\R^{2g}]$,
to establish Equation~\eqref{eq:local-rim-surgery-formula}, it is sufficient to show that
\begin{equation}\label{eq:reformulated-local-rim-surgery}
F_{\bar{\cW};\bar{\taus}} \circ F_{\cW';\taus'} \doteq
\Lef_{Z(\g,\omegas)}(C) \cdot F_{\bar{\cW};\bar{\taus}} \circ  F_{\cW;\taus},
\end{equation}
and that both sides are non-zero.

The sutured manifold cobordisms $\bar{\cW} \circ \cW$ and $\bar{\cW} \circ \cW'$
are equal to the complements of the tori $S^1 \times U$ and $C$ in
$S^1 \times S^3$, respectively. Write $\omegas_0$ for the tuple of 2-forms which is
$\taus$ on $\cW$ and $\bar{\taus}$ on $\bar{\cW}$. Write $\omegas_0'$ for the tuple of 2-forms
which is $\taus'$ on $\cW'$ and $\bar{\taus}$ on $\bar{\cW}$.
The proof of \cite{JZConcordanceSurgery}*{Proposition~5.3} implies that
\[
F_{\bar{\cW} \circ \cW; \omegas_0}(1) \doteq 1 \quad \text{and} \quad
F_{\bar{\cW} \circ \cW';\omegas'_0}(1) \doteq \Lef_{Z(\g,\omegas)}(C),
\]
which implies equation~\eqref{eq:reformulated-local-rim-surgery}, and hence
equation~\eqref{eq:local-rim-surgery-formula}, completing the proof.
\end{proof}

\begin{rem}
  If we use the alternate definition of concordance rim surgery from Remark~\ref{rem:rim},
  then one can directly invoke \cite[Corollary~5.5]{JZConcordanceSurgery}.
  Recall that the two definitions are equivalent by the work of Kim~\cite{Kim}*{Lemma~2.4} in the case
  of 1-twist rim surgery.
\end{rem}

\section{The invariant \texorpdfstring{$\Omega(S)$}{Omega(S)}}

In this section, we define an invariant $\Omega(S) \in \Z^{\ge 0} \cup \{-\infty\}$
for a surface $S \subset B^4$ with $g(S)>0$, bounding a knot
$K$ in $S^3$, and prove that it is a diffeomorphism invariant of $S$.

\subsection{Defining \texorpdfstring{$\Omega(S)$}{Omega(S)}}

Suppose $R$ is a UFD. We define
\[
\Omega_R \colon R \to \Z^{\ge 0} \cup \{-\infty\}
\]
to be the number of irreducible (non-unit) factors of an element of $R$,
counted with multiplicity. By convention, we set $\Omega(0) = -\infty$.

The map $\Omega$ may be extended to modules, in the following sense.
If $M$ is a free, finitely generated $R$-module, then we may define
\[
\Omega_M \colon M \to \Z^{\ge 0} \cup \{-\infty\},
\]
as follows. We set $\Omega(0) = -\infty$, and, for $x \neq 0$, we set
\[
\Omega_M(x) = \max\{\, \Omega_R(a) : x = a \cdot y, \text{ where } a \in R \text{, } y\in M \,\}.
\]
Equivalently, we may define $\Omega_M(x)$ by picking a free basis $e_1, \dots, e_n$ for $M$ over $R$,
writing $x = a_1 e_1+ \cdots +a_n e_n$, and setting
\[
\Omega_M(x) = \Omega_M(\gcd(a_1,\dots, a_n)).
\]
If $x \in M$ and $a \in R$, then
\begin{equation}\label{eq:additivity-of-Omega}
\Omega_M(a \cdot x) = \Omega_R(a) + \Omega_M(x).
\end{equation}
Note that, as $\Omega_M$ is defined without reference to a basis,
it is invariant under $R$-linear isomorphisms of $M$.

We focus now on the group ring $\bF[\R^n]$, which is not a UFD.  Note that
$\bF[\N^n] = \bF[z_1,\dots, z_n]$ is a UFD. Furthermore, $\bF[\Z^n]$ is also a
UFD, as it is the localization of $\bF[\N^n]$ at the set of monomials. Using
the above construction, we obtain a map $\Omega_{\bF[\Z^n]} \colon
\bF[\Z^n] \to \Z^{\ge 0} \cup \{-\infty\}$, as well as a similar function for
finitely generated, free modules over $\bF[\Z^n]$.

Suppose $p \in \bF[\R^n]$ is \emph{projectively integral} (Definition~\ref{def:proj-int}),
and $m$ is a monomial such that $m \cdot p \in \bF[\Z^n]$.
We then define
\[
\Omega_{\bF[\Z^n]}(p) := \Omega_{\bF[\Z^n]}(m \cdot p),
\]
which is clearly independent of the choice of monomial $m$.

More generally, suppose $M_0$ is a finite dimensional $\bF$-vector space, and
\[
M = M_0\otimes \bF[\R^n].
\]
If $x \in M$ is projectively integral,
and $m \cdot x \in M_0 \otimes \bF[\Z^n]$ for a monomial $m$, we define
\[
\Omega_M(x) := \Omega_{M_0 \otimes \bF[\Z^n]}(m \cdot x),
\]
which clearly does not depend on the monomial $m$.

If $f \colon \R^n \to \R^n$ is linear map, then there is an induced
endomorphism $\underline{f}$ on the group ring $\bF[\R^n]$, defined on
monomials via the formula
\[
\underline{f}(e^{\ve{a}}) = e^{f(\ve{a})}.
\]
It is straightforward to check the following:

\begin{lem}\label{lem:easy-algebra}
Suppose that $M_0$ is an $\bF$-vector space and $M = M_0 \otimes \bF[\R^n]$.
Suppose further that $\phi$ is an automorphism of $M_0$,
and $f\in \GL_n(\Z)$. If $x$ is projectively integral,
then $(\phi \otimes \underline{f})(x)$ is also projectively integral, and
\[
\Omega_M\left((\phi \otimes \underline{f})(x)\right) = \Omega_M(x).
\]
\end{lem}

\begin{define}
Let $S \subset B^4$ be an oriented surface of genus $g>0$ in $B^4$, bounding
a knot $K$ in $S^3$, and let $\omegas$ be a $2g$-tuple of integral 2-forms that form a
basis of $H^2(B^4 \setminus N(S), S^3 \setminus N(K))$. Let $\cS$ denote $S$
decorated with a single dividing arc, such that $g(\cS_{\ws}) = 0$ and
$g(\cS_{\zs}) = g(S)$. Proposition~\ref{prop:integrality} implies that
$\hat{F}_{B^4,\cS;\omegas}(1)$ is projectively integral. We define
\[
\Omega(S;\omegas):=\Omega_{\HFKh(K)\otimes \bF[\Z^{2g}]}(\hat{F}_{B^4,\cS;\omegas}(1)).
\]
\end{define}

\subsection{Diffeomorphism invariance of $\Omega$}

In this section, we prove that $\Omega(S;\omegas)$ is independent of the choice of $\omegas$,
and furthermore, it is a diffeomorphism invariant of $S$.

\begin{thm}\label{thm:diffeomorphism-invariant}
 Suppose $S$ is an oriented, genus $g > 0$ surface bounding a knot $K$ in $S^3$. Then
 the quantity $\Omega(S;\omegas) \in \Z^{\ge 0} \cup \{-\infty\}$ is
 independent of the choice of integral basis of 2-forms $\omegas$, and is a
 diffeomorphism invariant of the pair $(B^4,S)$.
\end{thm}

We begin by considering the dependence of $\Omega(S;\omegas)$ on the basis of integral 2-forms $\omegas$:

\begin{lem}\label{lem:transform-coefficients}
Let $S \subset B^4$ be a properly embedded surface of genus $g>0$,
bounding a knot $K$ in $S^3$. Suppose that $\omegas$ and $\omegas'$ are two $2g$-tuples of closed 2-forms
that both induce a basis of
\[
G := H^2(B^4 \setminus N(S), S^3\setminus N(K);\R)\iso \R^{2g},
\]
where $g = g(S)$.
Suppose $f \colon G \to G$ is an automorphism that sends $[\omegas]$ to $[\omegas']$.
Using the bases $\omegas$ and $\omegas'$, the map $f$ induces an automorphism of $\Z^{2g}$,
for which we also write $f$. With respect to these identifications, we have
\[
\hat{F}_{B^4,\cS;\omegas'}(1) \doteq (\id_{\HFKh(K)} \otimes \underline{f}^t) ( \hat{F}_{B^4,\cS;\omegas}(1)),
\]
where $f^t$ denotes the transpose of $f$.
\end{lem}

\begin{proof}
Write $\omegas = (\omega_1, \dots, \omega_{2g})$ and
$\omegas' = (\omega_1', \dots, \omega_{2g}')$.
Let $\cW$ be the sutured manifold cobordism complementary to $S$.
Write $f$ as a matrix $(a_{i,j})_{1 \le i,j \le 2g}$, such that
\[
\omega'_j = a_{1,j} \omega_1 + \cdots + a_{2g,j} \omega_{2g} + d \eta_j,
\]
where $\eta_j$ is a 1-form that vanishes on a neighborhood of $S^3 \setminus N(K)$.

We consider the totally twisted cobordism map
\[
\underline{F}_{\cW} \colon \bF \to \HFKh(S^3,K) \otimes_{\bF} \bF[(C_2/B_2)(B^4 \setminus N(S))],
\]
described in Section~\ref{sec:totally-twisted}.

If $(\Sigma,\as,\bs,w,z)$ is a Heegaard diagram for $(S^3,\bK)$, then we can write
\begin{equation}\label{eq:expand-totally-twisted-map}
\underline{F}_{\cW}(1) = \sum_{\xs \in \bT_{\a} \cap \bT_{\b}}
\xs \otimes \left(\sum_{i=1}^{n_{\xs}} e^{h_{\xs,i}} \right),
\end{equation}
where $h_{\xs, i}\in(C_2/B_2)(B^4 \setminus N(S))$,
and $n_{\xs} \in \Z^{\ge 0}$.
The perturbed map $F_{\cW;\omegas}$ is given by tensoring with $1\in \bF[\R^{2g}]_{\omegas}$, so equation~\eqref{eq:expand-totally-twisted-map} becomes
\begin{equation}
F_{\cW;\omegas}(1)=\sum_{\xs \in \bT_{\a} \cap \bT_{\b}}
\xs \otimes \left(\sum_{i=1}^{n_{\xs}} e^{(\int_{h_{\xs,i}} \omega_1,\dots, \int_{h_{\xs,i}} \omega_{2g})}  \right), \label{eq:perturbed-from-totally-twisted}
\end{equation}
and
\begin{equation}
F_{\cW;\omegas'}(1)=\sum_{\xs \in \bT_{\a} \cap \bT_{\b}}
\xs \otimes \left(\sum_{i=1}^{n_{\xs}} e^{(\int_{h_{\xs,i}} \omega'_1,\dots, \int_{h_{\xs,i}} \omega'_{2g})}  \right). \label{eq:perturbed-from-totally-twisted-w'}
\end{equation}

Let us write
\[
f_0(\omega_j) := \sum_{i=1}^{2g} a_{i,j} \omega_{i},
\]
so $f_0(\omega_j) = \omega_j'- d \eta_j$.
Equation~\eqref{eq:perturbed-from-totally-twisted} gives
\begin{equation}
\begin{split}
(\id \otimes \underline{f}^t)(F_{\cW;\omegas}(1))&=\sum_{\xs \in \bT_{\a} \cap \bT_{\b}}
\xs \otimes \left(\sum_{i=1}^{n_{\xs}} e^{f^t(\int_{h_{\xs,i}} \omega_1,\dots, \int_{h_{\xs,i}} \omega_{2g})}  \right)\\
&=\sum_{\xs \in \bT_{\a} \cap \bT_{\b}}
\xs \otimes \left(\sum_{i=1}^{n_{\xs}} e^{(\int_{h_{\xs,i}} f_0(\omega_1),\dots, \int_{h_{\xs,i}} f_0(\omega_{2g}))}  \right).
\end{split}
\label{eq:manipuate-perturbed-omega}
\end{equation}

The main claim is that equations~\eqref{eq:perturbed-from-totally-twisted-w'} and \eqref{eq:manipuate-perturbed-omega}
agree up to an overall factor of $e^{\ve{a}}$ for some $\ve{a} \in \R^{2g}$.
We will show that, if $ \xs \otimes e^{h}$ and $\xs' \otimes e^{h'}$
are summands of $\underline{F}_{\cW}(1)$, then
\begin{equation}\label{eq:evaluation-Kx-ind-of-x}
\int_{h} d\eta_j=\int_{h'} d\eta_j.
\end{equation}
Indeed, Lemma~\ref{lem:integrality-totally-twisted} shows that,
if $\phi \in \pi_2(\xs,\xs')$ is a disk on $(\Sigma,\as,\bs,w,z)$, then the 2-chain
\[
c := h - h' - \tilde{\cD}(\phi)
\]
is closed, and so $\int_c d \eta_j = 0$.
As $\eta_{j}$ vanishes on a neighborhood of $S^3 \setminus N(K)$, we have
\[
\int_{h} d\eta_j = \int_{h'} d \eta_j.
\]
Equation~\eqref{eq:evaluation-Kx-ind-of-x} follows,
and hence so does the main result.
\end{proof}

Combining Lemma~\ref{lem:easy-algebra} with Lemma~\ref{lem:transform-coefficients},
we obtain the following:

\begin{cor}\label{cor:number-factors-ind-omegas}
Let $K$ be a knot in $S^3$.
If $S$ is a smooth, genus $g>0$ surface in $B^4$ bounding $K$, then $\Omega(S;\omegas)$
is independent of the integral basis of 2-forms $\omegas$. We henceforth write just $\Omega(S)$.
\end{cor}

We now prove that $\Omega(S)$ is a diffeomorphism invariant:

\begin{proof}[Proof of Theorem~\ref{thm:diffeomorphism-invariant}]
Suppose that $S$ and $S'$ are two genus $g > 0$ surfaces in $B^4$ that bound a knot $K$ in $S^3$,
and $\Phi \colon (B^4,S) \to (B^4,S')$ is a diffeomorphism. Let $\Phi'$ denote the restriction of $\Phi$
to the complement of $S$. Let $\cS$ and $\cS'$ denote $S$ and $S'$ decorated with a single dividing arc,
with $\ws$-subsurfaces both equal to a bigon.
Diffeomorphism invariance of the cobordism maps implies that
\begin{equation}\label{eq:diffeomorphism-invariance}
\hat{F}_{B^4, \cS'; \Phi'_*\omegas}(1) = (\Phi|_{(S^3,K)*} \otimes \id_{\bF[\R^{2g}]})
\left( \hat{F}_{B^4, \cS; \omegas}(1)\right) .
\end{equation}
Hence $\Omega(S') = \Omega(S)$, by Lemma~\ref{lem:easy-algebra}.
\end{proof}

\subsection{Constructing non-diffeomorphic families of surfaces}

\begin{thm}\label{thm:not-diffeomorphic}
Suppose that $S \subset B^4$ is a properly embedded surface, $\g \subset S$ is a
homologically nontrivial simple closed curve, and $C$ is a self-concordance of a knot $J$
that has nontrivial knot Floer Lefschetz polynomial $\Lef_z(C)$.
If $\hat{F}_{B^4, \cS} \neq 0$, where $\cS$ is a decoration of $S$ such that $\cS_{\ws}$ is a bigon
and $\cS_{\zs}$ is a genus $g(S)$ subsurface,
then $(B^4,S)$ and $(B^4,S(\g,C))$ are not diffeomorphic.

More generally, if $\{C_n \colon n \in \N\}$ is a set of self-concordances
such that $\Lef_z(C_n)$ and $\Lef_z(C_m)$ have a different number of irreducible factors for $n \neq m$,
then $(B^4, S(\g, C_n))$ are pairwise non-diffeomorphic.
\end{thm}
\begin{proof}  Let $\omegas$ denote a collection of closed, integral 2-forms inducing a basis of
\[
G := H^2(B^4 \setminus N(S), S^3\setminus N(K);\R),
\]
and let $\omegas'$ denote a corresponding basis
of integral 2-forms inducing a basis of
\[
H^2(B^4 \setminus N(S(\g,C)), S^3 \setminus N(K);\R).
\]
By Theorem~\ref{thm:concordance-rim-surgery}, we have
\[
\hat{F}_{B^4, S(\g,C); \omegas'}(1) = \Lef_{Z(\g, \omegas)}(C) \cdot \hat{F}_{B^4, S; \omegas}(1).
\]

Since $\g$ is homologically non-trivial in $S$, the rim torus $T_\g$ is homologically
non-trivial in $B^4 \setminus N(S)$. As $\omegas$ induces a basis of $G$, the monomial $Z(\g, \omegas)$
is non-constant. Hence, since $\Lef_z(C)$ is non-trivial, so is $\Lef_{Z(\g, \omegas)}(C)$.
Since $\hat{F}_{B^4,\cS}\neq 0$, we have $\Omega(S)\ge 0$. Hence, Theorem~\ref{thm:concordance-rim-surgery} and equation~\eqref{eq:additivity-of-Omega} imply that if
 $\Lef_{Z(\g, \omegas)}(C) \not \doteq 1$, then
\begin{equation}\label{eq:inequality-irr-factors}
\Omega(S(\g,C)) >
\Omega ( S)\ge 0.
\end{equation}
Theorem~\ref{thm:diffeomorphism-invariant} implies that $S(\g,C)$ and $S$ are not diffeomorphic. This completes the proof of the first claim.

The second part follows similarly, as $\Omega(S(\g,C_n))$
are pairwise distinct.
\end{proof}

\section{Quasipositive knots, ascending surfaces, and Weinstein cobordisms}
\label{sec:constructible}

In this section, we provide some background on quasipositive knots, and ascending surfaces
in Weinstein cobordisms.

\subsection{Quasipositive links and braids}
\label{sec:background-quasipositive}
The braid group on $n$ strands has the presentation
\[
B_n=\langle\, \sigma_1,\dots, \sigma_{n-1} | \sigma_i \sigma_{i+1} \sigma_i = \sigma_{i+1} \sigma_i \sigma_{i+1},
\sigma_i \sigma_j = \sigma_j \sigma_i \text{ for } i-j \ge 2 \,\rangle,
\]
where $\sigma_i$ is the braid of a positive crossing between strands $i$ and $i+1$.

A link is \emph{positive} if it has only positive crossings.
A link is \emph{quasipositive} if it is isotopic to the closure
of a braid which can be represented as the product
\begin{equation}
\cW=\prod_{i=1}^k w_i \sigma_{j_i} w_i^{-1},\label{eq:quasipositive-def}
\end{equation}
for some $w_i\in B_n$. Finally, a link is \emph{strongly quasipositive} if
it is the closure of a braid that can be presented as a product of factors of the form
\begin{equation}\label{eq:strongly-quasipositive-def}
(\sigma_i\cdots \sigma_{j-2})(\sigma_{j-1})(\sigma_{i}\cdots \sigma_{j-2})^{-1},
\end{equation}
for $i \le j-2$.
Rudolph~\cite{Rudolph-pos-quasipos} proved that all positive links are strongly quasipositive.

If $K$ is a quasipositive link, which is the closure of a word $\cW$ with
presentation as in equation~\eqref{eq:quasipositive-def}, then there is an
induced link cobordism from the empty link to $K$. The link cobordism starts
with $n$ births, which give an unlink of $n$ components, which we view as the
closure of a trivial $n$-stranded braid. Each factor of the form $w_i
\sigma_{j_i} w_i^{-1}$ determines an isotopy, followed by a saddle. We call
this the \emph{quasipositive link cobordism} of the word $\cW$.

If each factor of $\cW$ has the form of equation~\eqref{eq:strongly-quasipositive-def},
then it gives rise to a Seifert surface for $K$ that we call
the \emph{quasipositive Seifert surface} of the word $\cW$. This is isotopic to the quasipositive link
cobordism of $\cW$ relative to $K$.

\subsection{Weinstein manifolds and ascending surfaces}

In this section, we recall background on Weinstein manifolds and a natural
family of smooth embedded surfaces therein, called \emph{ascending surfaces}.
See \cite{CEStein->Weinstein} for additional background. Our exposition
mostly  follows Hayden \cite{HaydenSteinSurfaces}.

\begin{define}
If $(X,J)$ is a complex manifold, we say that a map $\phi \colon X \to \R$
is \emph{$J$-convex} if the 2-form $\omega_\phi := -d (d^\C\phi)$ is
symplectic, where $d^\C \phi = d \phi \circ J$.  A compact, complex manifold
$(X,J)$ is called \emph{Stein} if it admits a $J$-convex function which is
exhausting (i.e., proper and bounded from below). If $X$ is compact and has
boundary $-Y_0\cup Y_1$,  we say that $(X,J)$ is a \emph{Stein cobordism}
from $Y_0$ to $Y_1$ if it admits a $J$-convex function $\phi \colon X \to [0,1]$
such that $0$ and $1$ are regular values, and $\phi^{-1}(0)=Y_0$ and $\phi^{-1}(1) = Y_1$.
Then $\xi_i := TY_i \cap J(TY_i)$ is a contact structure on $Y_i$ for $i \in \{0,1\}$.
\end{define}

\begin{define}
A \emph{Weinstein manifold} is a tuple $(X,\omega,\phi,V)$, where
$(X,\omega)$ is a symplectic manifold, $\phi$ is an exhausting Morse
function, and $V$ is a Liouville vector field (i.e., $\cL_V \omega = \omega$)
which is \emph{gradient-like} for $\phi$.
\end{define}

If $(X,J)$ is a Stein manifold with $J$-convex Morse function $\phi$,
then $(X, \omega_\phi, \phi, V_\phi)$ is a Weinstein manifold,
where $V_\phi$ is the gradient of $\phi$ with respect to the metric
$\langle X, Y \rangle := \omega_\phi(X, JY)$.
We will be interested in link cobordisms in Weinstein manifolds which satisfy the following property.

\begin{define}
A smoothly embedded surface $S$ in a Weinstein manifold $(X,\omega,\phi, V)$
is called \emph{ascending} if $S$ contains no critical points of $\phi$,
the restriction $\phi|_{S}$ is Morse, and $d\phi\wedge i_V \omega|_{S}>0$
away from the critical points of $\phi|_S$.
\end{define}

The definition is due to  Boileau--Orevkov
\cite{Boileau-Orevkov}*{D\'{e}finition 2} when $X=B^4$, and Hayden
\cite{HaydenSteinSurfaces}*{Definition~4.1} for general Stein and Weinstein
$X$.  We could equivalently require each regular level set
$\phi|_{S}^{-1}(c)$ to be a positive transverse link in the hypersurface $\phi^{-1}(c)$
with the contact form $i_V \omega$.

The tangent spaces at the critical points of ascending surfaces in Stein
manifolds are always $J$-complex lines; see Boileau--Orevkov
\cite{Boileau-Orevkov}*{p.~828} and Hayden
\cite{HaydenSteinSurfaces}*{Proposition~4.10}.  Hayden's proof adapts to show
that if $p\in S$ is a critical point of $\phi|_S$, then $i_V \omega$
restricts trivially to $T_p S$, which implies that $\omega$ restricts
non-trivially to $\Lambda^2 T_p S$.

\begin{define}
If $S$ is an ascending surface in a Weinstein manifold $(W,\omega,\phi,V)$,
we say that a critical points $p \in S$ of $\phi|_S$ is \emph{positive} (resp.~\emph{negative})
if the symplectic form restricts positively (resp.~negatively) to $T_p S$.
\end{define}

If $S$ is an ascending surface in a Stein manifold $(X,J)$ with $J$-convex Morse function $\phi$,
then a critical point $p \in S$ of $\phi|_S$ is positive precisely when
the orientation of $T_p S$ coincides with the complex orientation from $J$.
If $S$ is a $J$-holomorphic curve in a Stein manifold $(X,J)$, then $S$ is an ascending surface
with respect to a generic $J$-convex $\phi$ \cite{HaydenSteinSurfaces}*{Proposition~4.11}.

The following lemma seems to be well-known; however, we have been unable to locate a reference.

\begin{lem}\label{lem:ascending-surface}
Suppose that $(X,\omega, \phi,V)$ is Weinstein.
Suppose further that $S\subset X$ is an ascending surface and $p$ is a critical point
of $\phi|_S$. If $p$ has index 0, then $p$ is positive.
If $p$ has index 2, then $p$ is negative.
\end{lem}

\begin{proof}
We focus on the case when $p$ has index 0, since the argument is symmetric
when $p$ has index~2. Let $\epsilon>0$ be small. Note that shifting $\phi$ by
an overall constant preserves the Weinstein condition on $\phi$, and the
ascending condition on $S$, so assume $\phi(p)=-\epsilon$. Let
$D_\epsilon\subset S$ denote the component of $\phi|_S^{-1}([-\epsilon,0])$
containing $p$. Note that $\phi = 0$ on $\d D_\epsilon$.

Recall that the ascending condition on $S$ is equivalent to $\phi|_{S}$ being
Morse, and $d \phi\wedge i_V \omega|_S>0$ away from the critical points.
As we saw above, an easy adaptation of \cite{HaydenSteinSurfaces}*{Proposition~4.5}
gives that $\omega$ restricts non-trivially to $T_p S$.
Hence, it suffices to compute the sign.

Since $\phi=0$ on $\d D_\epsilon$, Stoke's theorem gives
\begin{equation}\label{eq:integral=0}
\int_{D_\epsilon} d(\phi i_V\omega)=\int_{\d D_\epsilon} \phi i_V\omega=0.
\end{equation}
On the other hand, the Leibniz rule gives
\begin{equation}\label{eq:Leibniz}
d (\phi i_V \omega)=d \phi\wedge i_V \omega+\phi \omega,
\end{equation}
since $d i_V \omega=\omega$ by $\cL_V \omega = \omega$.
Hence, combining equations~\eqref{eq:integral=0} and~\eqref{eq:Leibniz}, we obtain
\[
0 = \int_{D_\epsilon} d(\phi i_V \omega) =
\int_{D_\epsilon} d \phi\wedge i_V \omega + \int_{D_\epsilon} \phi \omega .
\]
Since $d\phi \wedge i_V \omega|_S > 0$ away from $p$,
we conclude that $\int_{D_\epsilon} \phi \omega < 0$.
Since $\phi\le 0$ near $p$, we conclude that $\omega$ must be positive near $p$.
\end{proof}

\subsection{Open books and transverse knots}

We begin with some background on open books and transverse knots in contact manifolds.

\begin{define}
\begin{enumerate}[label=($\cO$-\arabic*)]
\item An \emph{abstract open book} $\cO$ is a pair $(\page, h)$,
    where $\page$ is a compact, oriented surface with boundary,
    and $h\colon \page\to \page$ is an automorphism such that $h|_{\d \page}=\id$.
    Note that $\cO$ determines a 3-manifold $Y_{\cO} = (\page\times I)/{\sim}$, where
    $(x,1) \sim (h(x),0)$ for all $x \in \page $, and $ (x,t) \sim (x, t')$
    whenever $t$, $t' \in I$ and  $x \in \d \page$.
\item An \emph{embedded open book} $(\page, h, \phi)$ of $Y$ consists of
    an abstract open book $\cO = (\page, h)$, and a diffeomorphism
    \[
    \varphi \colon Y_{\cO}\to Y.
    \]
    We say that $(\page, h, \phi)$ is \emph{compatible} with a co-oriented
    contact structure $\xi$ if $\xi$ is isotopic to a contact structure
    which is the kernel of a 1-form $\a$ such that $d\a > 0$ on the interior
    of the pages of the open book, and $\a > 0$ on the binding.
\end{enumerate}
\end{define}

\begin{define}\label{def:transverse}
\begin{enumerate}[label=($\cT$-\arabic*)]
\item Suppose that $(Y,\xi)$ is a co-oriented contact 3-manifold.
    A link $L$ in $(Y,\xi)$ is \emph{transverse} if $T_p L$
    is transverse to $\xi_p$ for all $p \in L$. If $L$ is oriented,
    we say that $L$ is a \emph{positive} transverse link if $L$ is positively transverse to $\xi$.
\item An \emph{abstract, pointed open book} $\cO$ is a tuple $(\page,h,\ve{p})$,
    where $(\page, h)$ is an open book, and $\ve{p} \subset \Int(\page)$
    is a finite collection of points which is fixed by $h$ setwise.
    Note that a pointed open book determines a 3-manifold $Y_{\cO}$ containing a link $L_\cO$.
\item If $\bL = (L,\ws,\zs)$ is a multi-pointed link in a 3-manifold $Y$,
    we say an \emph{embedded, pointed open book} of $(Y,\bL)$ is a tuple
    $(\page, h, \ve{p}, \phi)$ such that $(\page, h, \phi)$ is
    an embedded open book for $Y$, and
\begin{enumerate}
\item $\varphi(\ve{p}\times I) = L$,
\item $\varphi(\ve{p}\times \{1/2\}) = \ws$ and $\varphi(\ve{p}\times \{0\}) = \zs$.
\end{enumerate}
\item\label{it:compat} If $\bL = (L,\ws,\zs)$ is a multi-pointed, transverse link in $(Y,\xi)$,
    we say that an embedded, pointed open book for $(Y,\bL)$ is \emph{compatible}
    with $\xi$ if there is an isotopy of contact 1-forms $\a_t$ such that
\begin{enumerate}
\item $\xi = \ker \a_0$,
\item $\a_t > 0$ on $L$ for each $t$,
\item $d \a_1 > 0$ on the pages, and $\a_1 > 0$ on the binding.
\end{enumerate}

\end{enumerate}
\end{define}

\begin{define}\label{def:arcs}
Suppose that $\cO = (\page, h, \ve{p})$ is a pointed open book.
A \emph{basis of arcs} of $\cO$ consists of a collection of pairwise disjoint arcs $\ve{a} = (a_1, \dots, a_n)$,
such that each component of $\page  \setminus \ve{a}$ is a disk containing exactly one point of $\ve{p}$.
\end{define}

\subsection{Partial open books and transverse knots}
\label{sec:partial-open-books}

In this section, extending work of Stipsicz and V\'ertesi~\cite{StipsiczVertesiEH-LOSS}
for doubly-pointed transverse knots, we construct a contact structure on the sutured manifold complementary
to a multi-pointed transverse link by attaching bypasses to the complement of a Legendrian approximation,
and relate the corresponding partial open books. We will use this for showing the naturality
of the transverse link invariants in Section~\ref{sec:naturality}.

The following definition is due to Honda, Kazez, and Mati\'c~\cite{HKMSutured}:

\begin{define}
An \emph{abstract partial open book} $\cO$ is a triple $\cO = (\page, P, h)$, where
$\page$ is a compact, connected, oriented surface with non-empty boundary,
$P \subset \page$ is a compact subsurface such that $\page$
is obtained from $\cl(\page \setminus P)$ by successively attaching 1-handles,
and $h \colon P \to \page$ is an embedding such that $h|_{P\cap \d \page } = \id$.
\end{define}

A partial open book $\cO$ naturally determines a 3-manifold manifold
$M_{\cO} = (\page \times I)/{\sim}$, where $(x,1)\sim (h(x),0)$ for $x\in P$,
and $(x,t) \sim (x,t')$ if $x \in \d \page \cap \d P$ and $t$, $t' \in I$, with
sutures $\g_{\cO} = \d(\page \setminus P) \times \{1\}$. If $(M,\g)$ is a sutured
manifold, then an \emph{embedded} partial open book is a partial open book
$\cO$ equipped with a diffeomorphism $\varphi \colon (M_{\cO}, \g_{\cO}) \to (M, \g)$.

If $\bL=(L,\ve{w},\ve{z})$ is a multi-pointed link in $Y$, we will write
$Y(\bL)$ for the complementary sutured manifold obtained by removing a neighborhood of $L$,
and adding meridional sutures corresponding to the basepoints of $\bL$.

Suppose that $\bL=(L,\ws,\zs)$ is a multi-pointed transverse link in $(Y,\xi)$. Let $\cL$ be a
Legendrian approximation of $L$. Choose a standard neighborhood $N(\cL)$ of $\cL$, and write
$\xi_0 = \xi|_{Y \setminus N(\cL)}$. The dividing set on $\d (Y \setminus N(\cL))$
consists of two copies of the contact longitude on each component of $L$.

Extending the construction of Stipsicz and V\'ertesi~\cite{StipsiczVertesiEH-LOSS}
from doubly-pointed to multi-pointed knots, we attach a collection of bypasses
to $\xi_0$ along $\d (Y \setminus N(\cL))$ to obtain a contact structure
$\xi_{\bL}^{\pitchfork}$ on $Y(\bL)$. If $K$ is a component of $\bL$
that has $2n$ basepoints, then we attach $n$ bypasses along $\d N(K)$, as in Figure~\ref{fig:19}.

\begin{figure}[ht!]
	\centering
\begingroup%
  \makeatletter%
  \providecommand\color[2][]{%
    \errmessage{(Inkscape) Color is used for the text in Inkscape, but the package 'color.sty' is not loaded}%
    \renewcommand\color[2][]{}%
  }%
  \providecommand\transparent[1]{%
    \errmessage{(Inkscape) Transparency is used (non-zero) for the text in Inkscape, but the package 'transparent.sty' is not loaded}%
    \renewcommand\transparent[1]{}%
  }%
  \providecommand\rotatebox[2]{#2}%
  \newcommand*\fsize{\dimexpr\f@size pt\relax}%
  \newcommand*\lineheight[1]{\fontsize{\fsize}{#1\fsize}\selectfont}%
  \ifx\svgwidth\undefined%
    \setlength{\unitlength}{150.0085582bp}%
    \ifx\svgscale\undefined%
      \relax%
    \else%
      \setlength{\unitlength}{\unitlength * \real{\svgscale}}%
    \fi%
  \else%
    \setlength{\unitlength}{\svgwidth}%
  \fi%
  \global\let\svgwidth\undefined%
  \global\let\svgscale\undefined%
  \makeatother%
  \begin{picture}(1,0.76917764)%
    \lineheight{1}%
    \setlength\tabcolsep{0pt}%
    \put(0,0){\includegraphics[width=\unitlength,page=1]{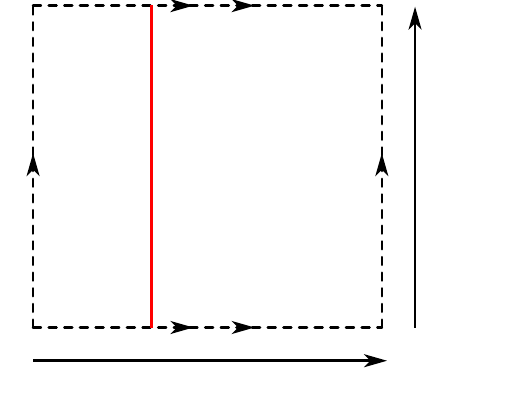}}%
    \put(0.40033348,0.00986098){\color[rgb]{0,0,0}\makebox(0,0)[rt]{\lineheight{1.25}\smash{\begin{tabular}[t]{r}$S^1_\mu$\end{tabular}}}}%
    \put(0.82694288,0.35516407){\color[rgb]{0,0,0}\makebox(0,0)[lt]{\lineheight{1.25}\smash{\begin{tabular}[t]{l}$S^1_{\ell}$\end{tabular}}}}%
    \put(0,0){\includegraphics[width=\unitlength,page=2]{fig22.pdf}}%
    \put(0.40184783,0.18427816){\color[rgb]{0,0,0}\makebox(0,0)[t]{\lineheight{1.25}\smash{\begin{tabular}[t]{c}$R_+$\end{tabular}}}}%
    \put(0.62270584,0.18427816){\color[rgb]{0,0,0}\makebox(0,0)[t]{\lineheight{1.25}\smash{\begin{tabular}[t]{c}$R_-$\end{tabular}}}}%
    \put(0,0){\includegraphics[width=\unitlength,page=3]{fig22.pdf}}%
  \end{picture}%
\endgroup%

	\caption{Bypass attachments along $\d(Y \setminus N(L))$
    to obtain $\xi_{\bL}^{\pitchfork}$ from $\xi_0$.}
\label{fig:19}
\end{figure}

Following the proof of \cite{StipsiczVertesiEH-LOSS}*{Theorem~1.5}, we see
that the contact structure $\xi_{\bL}^{\pitchfork}$ on $Y \setminus N(L)$ is
invariant under negative stabilizations of $\cL$, up to isotopy relative to
$\d(Y \setminus N(L))$, in the following sense. If $\cL^-$ is a negative stabilization, we
may view $\cL^-$ as lying in $N(L)$. Under the canonical diffeomorphism
between $Y\setminus N(\cL)$ and $Y \setminus N(\cL^-)$, the contact structures
constructed from $\cL$ and $\cL^-$ are isotopic, relative to $\d(Y \setminus N(L))$.

Given a pointed open book $\cO=(\page ,\ve{p}, h)$ for a multi-pointed link
$\bL$ in $Y$, we may obtain a partial open book
$\cO^\circ:=(\page^\circ,P,h|_P)$ for the sutured manifold $Y(\bL)$, as
follows. Take $\page^\circ = \page \setminus N(\ve{p})$. The subsurface $P$ is obtained from
$\page^\circ$ by removing a neighborhood of $|\ps|$ pairwise disjoint arcs, each connecting a
point of $\ps$ to $\d \page $.  The partial open book $\cO^\circ$ is clearly
a partial open book for $Y(\bL)$. Moreover, we have the following:

\begin{prop} \label{prop:pointed-to-partial}
Let $\bL$ be a multi-pointed transverse link in the closed contact 3-manifold $(Y, \xi)$.
If $(\cO, \varphi)$ is an embedded, pointed open book for $(Y,\bL)$ compatible
with $\xi$ in the sense of part \ref{it:compat} of Definition~\ref{def:transverse},
then the partial open book $(\cO^\circ, \varphi|_{Y_{\cO^\circ}})$
for $Y(\bL)$ is compatible with $\xi_{\bL}^\pitchfork$.
\end{prop}

\begin{proof}
The proof is by induction on the number of basepoints per component. The case
when each component has two basepoints follows from the work of
Stipsicz and V\'ertesi~\cite{StipsiczVertesiEH-LOSS}.
To obtain the statement more generally, we
argue by induction. To increase the number of basepoints on a pointed open
book, we may perform a positive Markov stabilization to $\cO$, as described
by Baldwin, Vela-Vick, and V\'ertesi \cite{BVVVTransverse}*{Section~2.5};
see the left of Figure~\ref{fig:9}. To see that this has the same effect
as one of our bypass attachments, we argue as follows.
Bypasses may be attached one at a time. In Figure~\ref{fig:28}, we
illustrate one of the attaching arcs of a bypass used in the construction of
$\xi^{\pitchfork}_{\bL}$, after one bypass has already been attached. A
bypass attachment is, by definition, a contact 1-handle followed by a contact
2-handle. The feet of the contact 1-handle are attached at the end points of
the bypass arc, and the contact 2-handle is attached along the union of a
longitude of the 1-handle, and the original bypass arc.  The effect of
contact 1-handles and 2-handles on partial open books is described in
\cite{JuhaszZemkeContactHandles}*{Figure~1.1}. An easy local computation
shows the effect on partial open books of the bypass shown on the right-hand side
of Figure~\ref{fig:28} coincides with a positive Markov stabilization of
pointed open books.
\end{proof}

\begin{figure}[ht!]
	\centering
	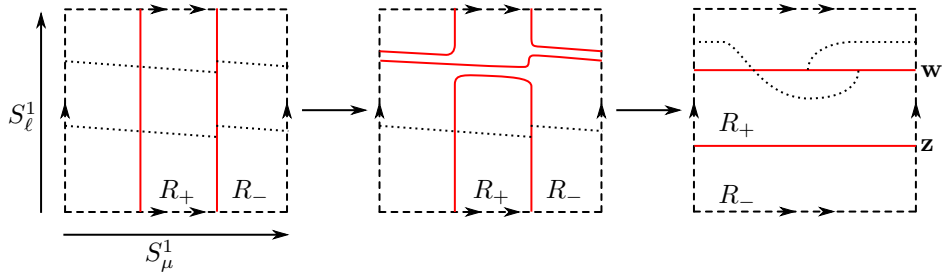
	\caption{Left: Two bypasses giving $\xi_{\bL}^{\pitchfork}.$
    Middle: The effect of attaching the top bypass. Right: An isotopy of the middle picture.
    The sutures on the right are labeled by the corresponding basepoint.}
\label{fig:28}
\end{figure}

\subsection{Ascending surfaces and open books}

In this section, we recall how to interpret ascending surfaces in terms of
pointed open books. Given a contact manifold $(Y, \xi)$, we denote its symplectization
by $\Symp(Y,\xi)$. We first consider index 0 critical points of an ascending
surface, which are always positive by Lemma~\ref{lem:ascending-surface}.

\begin{lem}\label{lem:open-books-index-0}
Suppose that $(W,\cS)\colon (Y,\bL_0,\xi)\to (Y,\bL_1,\xi)$ is an ascending
link cobordism with $W = [a,b] \times Y \subset \Symp(Y,\xi)$, such that $\cS$
has a single (positive) index 0 critical point $p$, and $\cS$ is locally
translation invariant outside of a small neighborhood of $p$, and the
basepoints of $\bL_1$ are the images of the basepoints of $\bL_0$ under the
translation in the $[a,b]$-direction. If $\cO = (\page,\ve{p},h)$ is a pointed
open book for $(Y,\bL_0,\xi)$, then a pointed open book for $(Y,\bL_1,\xi)$
may be obtained (after an ambient isotopy of $\cO$) by adding a basepoint
near the boundary of $\page$, and keeping $h$ unchanged.
\end{lem}

\begin{proof}
Suppose $(c, p) \in [a,b] \times Y$ is the critical point of $\cS$.
Isotope $\cO$ so that the binding passes through $p$,
transverse to $\xi_p$. Then the regular level sets of $\cS$ immediately
after $c$ intersect the pages of the open book in a single point, since
$T_{(c, p)} S = \{c\} \times \xi_p$. Let the intersection of the level
set of $S$ with value $c+\epsilon$ and a page of the
open book be the new basepoint. It is straightforward to see that this open
book has the desired properties.
\end{proof}

We now consider index 1 critical points. We recall that, if
$p_1$ and $p_2$ are distinct points on a surface $\Sigma$ connected by
an embedded path $\gamma$, there are half-twist diffeomorphisms
$\tw_{\g}^{\pm}$ that swap $p_1$ and $p_2$, and are supported in a
neighborhood of $\g$, as in Figure~\ref{fig:2}.

\begin{lem}\label{lem:ascending-index-1-half-twist}
 Suppose that $(W, \cS) \colon (Y,\bL_0,\xi) \to (Y,\bL_1,\xi)$ is an ascending
 link cobordism with $W = [a,b] \times Y \subset \Symp(Y,\xi)$, such that $\cS$
 has a single index 1 critical point, of sign $\epsilon = \pm 1$.  Suppose
 further that, outside a small neighborhood of the critical point, the surface
 $\cS$ is locally translation invariant, and the basepoints of $\bL_1$ coincide
 with the images of the basepoints of $\bL_0$ under translation. Then we may
 pick a pointed open book $(\page,\ve{p},h)$ for $(Y,\bL_0,\xi)$ such that
 $(\page,\ve{p},\tw_{\g}^{\epsilon} \circ  h)$ is a pointed open book for
 $(Y,\bL_1,\xi)$, where $\tw_{\g}^{\epsilon}$ is a half-twist diffeomorphism
 associated to a path on $\page$ connecting two distinct basepoints, of sign
 $\epsilon$.
\end{lem}

A proof of Lemma~\ref{lem:ascending-index-1-half-twist}
may be found by following the proof of \cite{HaydenSteinSurfaces}*{Theorem~4.3}; see specifically Lemma~3.5 therein.

\begin{figure}[ht!]
\centering
\begingroup%
  \makeatletter%
  \providecommand\color[2][]{%
    \errmessage{(Inkscape) Color is used for the text in Inkscape, but the package 'color.sty' is not loaded}%
    \renewcommand\color[2][]{}%
  }%
  \providecommand\transparent[1]{%
    \errmessage{(Inkscape) Transparency is used (non-zero) for the text in Inkscape, but the package 'transparent.sty' is not loaded}%
    \renewcommand\transparent[1]{}%
  }%
  \providecommand\rotatebox[2]{#2}%
  \newcommand*\fsize{\dimexpr\f@size pt\relax}%
  \newcommand*\lineheight[1]{\fontsize{\fsize}{#1\fsize}\selectfont}%
  \ifx\svgwidth\undefined%
    \setlength{\unitlength}{301.88948294bp}%
    \ifx\svgscale\undefined%
      \relax%
    \else%
      \setlength{\unitlength}{\unitlength * \real{\svgscale}}%
    \fi%
  \else%
    \setlength{\unitlength}{\svgwidth}%
  \fi%
  \global\let\svgwidth\undefined%
  \global\let\svgscale\undefined%
  \makeatother%
  \begin{picture}(1,0.3006937)%
    \lineheight{1}%
    \setlength\tabcolsep{0pt}%
    \put(0,0){\includegraphics[width=\unitlength,page=1]{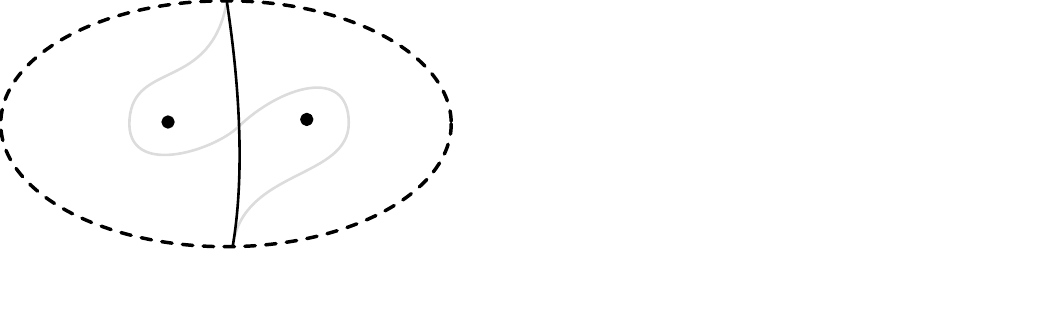}}%
    \put(0.21244993,0.00698237){\color[rgb]{0,0,0}\makebox(0,0)[t]{\lineheight{1.25}\smash{\begin{tabular}[t]{c}$\tw^+_\g$\end{tabular}}}}%
    \put(0,0){\includegraphics[width=\unitlength,page=2]{fig2.pdf}}%
    \put(0.78124854,0.00453437){\color[rgb]{0,0,0}\makebox(0,0)[t]{\lineheight{1.25}\smash{\begin{tabular}[t]{c}$\tw^-_\g$\end{tabular}}}}%
    \put(0,0){\includegraphics[width=\unitlength,page=3]{fig2.pdf}}%
    \put(0.25806559,0.15950077){\color[rgb]{0,0,0}\makebox(0,0)[t]{\lineheight{1.25}\smash{\begin{tabular}[t]{c}$\g$\end{tabular}}}}%
    \put(0.21726598,0.10087323){\color[rgb]{0,0,0}\makebox(0,0)[rt]{\lineheight{1.25}\smash{\begin{tabular}[t]{r}$\ve{a}$\end{tabular}}}}%
  \end{picture}%
\endgroup%

\caption{The positive and negative half-twist diffeomorphisms.}
\label{fig:2}
\end{figure}

\section{Transverse invariants, and their functoriality under ascending Weinstein cobordisms}
\label{sec:functoriality}

Lisca, Ozsv\'{a}th, Stipsicz, and Szab\'{o} \cite{LOSS} described an invariant
of Legendrian knots, which is commonly referred to as the \emph{LOSS}
invariant. The LOSS invariant is unchanged by negative stabilizations, and
hence gives an invariant of transverse knots, by work of Epstein, Fuchs, and
Meyer \cite{EpsteinFuchsMeyer} and Etnyre and Honda
\cite{EtynreHondaTransverse}. Baldwin, Vela-Vick, and
V\'ertesi~\cite{BVVVTransverse} constructed an extension of the LOSS
invariant using multi-pointed open books that they called the \emph{BRAID} invariant,
which is more immediately an invariant of transverse links. We will refer to the common invariant as the
\emph{transverse invariant}.   If $\bL=(L,\ws,\zs)$ is a transverse,
multi-pointed link in the contact 3-manifold $(Y,\xi)$, the transverse invariants take the form of
elements
\begin{equation}\label{eq:transverse-invariants-domain}
\hat{\tra}(Y,\bL,\xi)\in \HFLh(-Y,-\bL) \quad \text{and}  \quad
\tra^-(Y,\bL,\xi)\in \HFL^-(-Y,-\bL).
\end{equation}
See Remark~\ref{rem:conventions} for a note on orientation conventions.
We will review the construction of the transverse invariant in Section~\ref{sec:construction-braid}.

In this section, we study the functoriality of the transverse invariants.
We restate our main result, Theorem~\ref{thm:ascending-Stein} from the introduction,
after the following definition.

\begin{define}\label{def:arboreal-decoration}
Suppose $(W,\cS) \colon (Y_0,\bL_0) \to (Y_1,\bL_1)$ is a decorated link
cobordism, equipped with a Morse function $\phi \colon S \to I$.
Write $\bL_i = (L_i,\ws_i,\zs_i)$. We say $\cS$ is \emph{$\ws$-arboreal} with
respect to $\phi$ if there is a forest $F \subset \cS_{\ws}$ (i.e., a disjoint
union of trees) such that the following hold:
\begin{enumerate}
\item $\cS_{\ws}$ is a regular neighborhood of $F$.
\item $d\phi$ restricts non-trivially to each edge of $F$.
\item Each vertex $v$ of $F\cap \Int(W)$ is incident to exactly one edge on which $\phi$ takes
    lower value than $\phi(v)$.
\end{enumerate}
\end{define}

We say that a decorated link cobordism $(W,\cS)\colon (Y_0,\bL_0)\to (Y_1,\bL_1)$ is
\emph{$\ws$-anti-arboreal} with respect to a Morse function $\phi \colon S \to I$ if the
reversed link cobordism $(W,\cS)^{\vee}\colon (-Y_1,-\bL_1)\to
(-Y_0,-\bL_0)$ is $\ws$-arboreal with respect to $-\phi$. Note that the
composition of $\ws$-arboreal decorations is $\ws$-arboreal. See
Figure~\ref{fig:12} for an example of a $\ws$-anti-arboreal decoration.

\begin{figure}[ht!]
\centering
\begingroup%
  \makeatletter%
  \providecommand\color[2][]{%
    \errmessage{(Inkscape) Color is used for the text in Inkscape, but the package 'color.sty' is not loaded}%
    \renewcommand\color[2][]{}%
  }%
  \providecommand\transparent[1]{%
    \errmessage{(Inkscape) Transparency is used (non-zero) for the text in Inkscape, but the package 'transparent.sty' is not loaded}%
    \renewcommand\transparent[1]{}%
  }%
  \providecommand\rotatebox[2]{#2}%
  \newcommand*\fsize{\dimexpr\f@size pt\relax}%
  \newcommand*\lineheight[1]{\fontsize{\fsize}{#1\fsize}\selectfont}%
  \ifx\svgwidth\undefined%
    \setlength{\unitlength}{158.5815267bp}%
    \ifx\svgscale\undefined%
      \relax%
    \else%
      \setlength{\unitlength}{\unitlength * \real{\svgscale}}%
    \fi%
  \else%
    \setlength{\unitlength}{\svgwidth}%
  \fi%
  \global\let\svgwidth\undefined%
  \global\let\svgscale\undefined%
  \makeatother%
  \begin{picture}(1,0.64309144)%
    \lineheight{1}%
    \setlength\tabcolsep{0pt}%
    \put(0.12729832,0.36142802){\color[rgb]{0,0,0}\makebox(0,0)[rt]{\lineheight{1.25}\smash{\begin{tabular}[t]{r}$\bL_0$\end{tabular}}}}%
    \put(0.87253516,0.36142802){\color[rgb]{0,0,0}\makebox(0,0)[lt]{\lineheight{1.25}\smash{\begin{tabular}[t]{l}$\bL_1$\end{tabular}}}}%
    \put(0,0){\includegraphics[width=\unitlength,page=1]{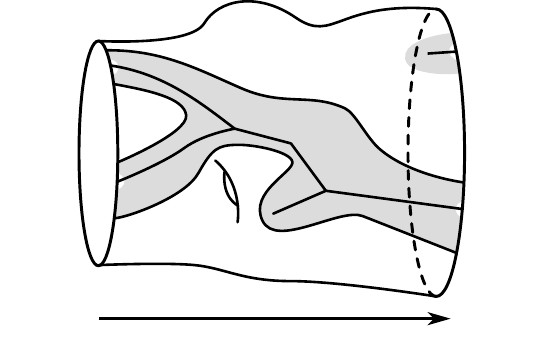}}%
    \put(0.57964463,0.51797032){\color[rgb]{0,0,0}\makebox(0,0)[lt]{\lineheight{1.25}\smash{\begin{tabular}[t]{l}$\cS_{\zs}$\end{tabular}}}}%
    \put(0.56177534,0.35908015){\color[rgb]{0,0,0}\makebox(0,0)[lt]{\lineheight{1.25}\smash{\begin{tabular}[t]{l}$\cS_{\ws}$\end{tabular}}}}%
    \put(0.44944172,0.41304586){\color[rgb]{0,0,0}\makebox(0,0)[lt]{\lineheight{1.25}\smash{\begin{tabular}[t]{l}$F$\end{tabular}}}}%
    \put(0,0){\includegraphics[width=\unitlength,page=2]{fig23.pdf}}%
    \put(0.46827627,0.00507428){\color[rgb]{0,0,0}\makebox(0,0)[t]{\lineheight{1.25}\smash{\begin{tabular}[t]{c}$\phi$\end{tabular}}}}%
  \end{picture}%
\endgroup%

\caption{A $\ws$-anti-arboreal decoration of a cobordism from $\bL_0$ to $\bL_1$.}
\label{fig:12}
\end{figure}

The main theorem of this section is the following:

\begin{customthm}{\ref{thm:ascending-Stein}}
Suppose $(W,\cS)\colon (Y_0,\bL_0)\to (Y_1,\bL_1)$ is a decorated link
cobordism such that $W$ has a Weinstein structure $(W,\omega,\phi,V)$,
and $\cS=(S,\cA)$, where $S$ is an ascending surface with positive critical
points. If the decoration $\cA$ is $\ws$-anti-arboreal with respect to
$\phi$, then
\[
\left(F^\circ_{W,\cS}\right)^{\vee}\left(\tra^\circ(\bL_1)\right)=\tra^\circ(\bL_0),
\]
for $\circ\in \{\wedge,-\}$.
\end{customthm}

\begin{rem}\label{rem:conventions}
Our orientation convention in equation~\eqref{eq:transverse-invariants-domain} departs
slightly from the conventions of Lisca, Ozsv\'ath, Stipsicz, and
Szab\'o~\cite{LOSS}, and Baldwin, Vela-Vick, and V\'ertesi~\cite{BVVVTransverse},
where the transverse invariants lie in $\HFL^\circ(-Y,\bL)$. Note, however,
that there is an isomorphism $\HFL^\circ(-Y,\bL) \iso \HFL^\circ(-Y,-\bL)$.
Ozsv\'ath, Szab\'o, and Thurston~\cite{OSTLegendrian} define two transverse invariants,
denoted $\lambda^+$ and $\lambda^-$.  The invariant $\lambda^+$ lies in
$\HFL^\circ(-Y,\bL)$, and coincides with the LOSS/BRAID invariants, as
defined in \cite{LOSS} and \cite{BVVVTransverse}, according to
\cite{BVVVTransverse}*{Theorem~8.1}. The invariant $\lambda^-$ lies in
$\HFL^\circ(-Y,-\bL)$, and coincides with the invariant for which we write
$\tra^\circ(Y,\bL,\xi)$.  Theorem~\ref{thm:ascending-Stein} could be stated
for the invariants in $\HFL^\circ(-Y,\bL)$ by considering $\zs$-anti-arboreal
decorations and the cobordism map $F_{W,-\cS}^\circ$, where $-\cS$ denotes
$\cS$ with the opposite orientation.
\end{rem}

We have the following immediate corollaries of Theorem~\ref{thm:ascending-Stein}:

\begin{cor}\label{cor:BRAID}
Let the knot $K$ in $S^3$ be the closure of a quasipositive braid, and let
$S$ be the corresponding quasipositive link cobordism.
We decorate $\cS$ with a dividing set consisting of a single arc, so that $g(\cS_{\ws}) = 0$ and
$g(\cS_{\zs}) = g(S)$.  We view $(B^4,\cS)$ as a cobordism from $\emptyset$ to
$(S^3,K)$. As $K$ is in braid position, it is a transverse knot in the
standard tight contact structure on $S^3$. We have
\[
\left(\hat{F}_{B^4, \cS}\right)^\vee\left(\hat{\tra}(K)\right) = 1 \in
\widehat{\HFK}(\emptyset)\iso \bF_2 \quad \text{and} \quad
\left(F_{B^4,\cS}^-\right)^{\vee}\left(\tra^-(K)\right) = 1 \in \HFK^-(\emptyset) \cong \bF[v].
\]
\end{cor}

\begin{cor}\label{cor:non-vanishing}
Suppose that $\cS\subset B^4$ is a decorated surface obtained
by pushing a quasipositive Seifert surface of a knot $K$ into $B^4$,
with a decoration such that $g(\cS_{\ws}) = 0$ and $g(\cS_{\zs}) = g(S)$. Then
\[
F^\circ_{B^4, \cS} \colon \HFK^\circ(\emptyset) \to \HFK^\circ(K)
\]
are non-vanishing, for $\circ \in \{\wedge,-\}$.
\end{cor}

We now state a stronger version of Theorem~\ref{thm:ascending-Stein}
for the cobordism maps with perturbed coefficients, which we will
prove at the end of Section~\ref{sec:proofs}.

\begin{prop}\label{prop:ascending-perturbed-coefficients}
 Suppose that $(W,\cS)\colon (Y_0,\bL_0)\to (Y_1,\bL_1)$ is a decorated link
 cobordism, with $W$ Weinstein, and $\cS$ an ascending surface with positive
 critical points and $\ws$-anti-arboreal decoration. Suppose further that
 $\omegas$ is an $n$-tuple of closed 2-forms on the complement of $\cS$ in
 $W$ that vanish on $Y_0$ and $Y_1$. Then,
 \[
	 \left(\hat{F}_{W,\cS;\omegas}\right)^\vee\left(\hat{\tra}(\bL_1)\otimes 1\right)\doteq \hat{\tra}(\bL_0)\otimes 1,
 \]
 as elements of $\HFLh(-Y_0,-\bL_0)\otimes \bF[\R^n]$. In particular, if
 $K\subset S^3$ is the closure of a quasipositive braid and $S\subset B^4$ is
 the associated ascending surface, then $\Omega(S)=0$.
\end{prop}

\subsection{The transverse invariants}
\label{sec:construction-braid}

We now review the definition of the transverse invariants, following
Baldwin, Vela-Vick, and V\'ertesi~\cite{BVVVTransverse}.
Suppose that $\bL = (L,\ws,\zs)$ is a multi-pointed,
transverse link in the contact 3-manifold $(Y,\xi)$. We pick an embedded, pointed open book
$\cO = (S,h,\ve{p})$ for $(Y,\bL,\xi)$.  We construct a multi-pointed link
diagram
\[
\cH(\cO) := (\Sigma, \as, \bs, \ws, \zs)
\]
for $(Y,\bL)$, as follows. We define
\[
\Sigma := \page  \times \{1/2\} \cup \bar{\page } \times \{0\}.
\]
Let $\ve{a} = \{a_1, \dots, a_n\}$ be a basis of arcs for $\cO$, as in Definition~\ref{def:arcs}.
For $i \in \{1,\dots,n\}$, let $b_i$ be a small translate of $a_i$ such that we move $\d a_i$
positively along $\d \page $, and such that there is a single transverse intersection point
between $a_i$ and $b_i$. We define
\[
\alpha_i := a_i \times \{1/2\} \cup a_i \times \{0\} \quad \text{and} \quad \beta_i:=b_i \times
\{1/2\} \cup h(b_i) \times \{0\}.
\]
Furthermore, let
$\ws = \ps \times \{1/2\} \in \page \times \{\tfrac{1}{2}\} $ and
$\zs = \ps \times \{0\} \in \bar{\page } \times \{0\}$.
Finally, set
\[
\cH^\vee(\cO):=(\Sigma, \bs, \as, \ws, \zs),
\]
which is a diagram for $(-Y,-\bL)$.

For $\circ \in \{\wedge,-\}$, the transverse invariants
\begin{equation}\label{eq:transervse-invariants-def}
\hat{\frT}(Y,\bL,\xi) \in \HFLh(\cH^\vee(\cO))\quad \text{and} \quad \frT^-(Y,\bL,\xi)\in \HFL^- (\cH^\vee(\cO))
\end{equation}
are given by the homology class of the intersection point
\[
\xs(\cO) := (x_1, \dots, x_n),
\]
where $x_i = a_i \times \{1/2\} \cap b_i \times \{1/2\}$.
See Figure~\ref{fig:15} for an example of a planar open book $\cO$,
the Heegaard diagram $\cH^\vee(\cO)$, and the intersection point $\xs(\cO)$.

\begin{figure}[ht!]
	\centering
	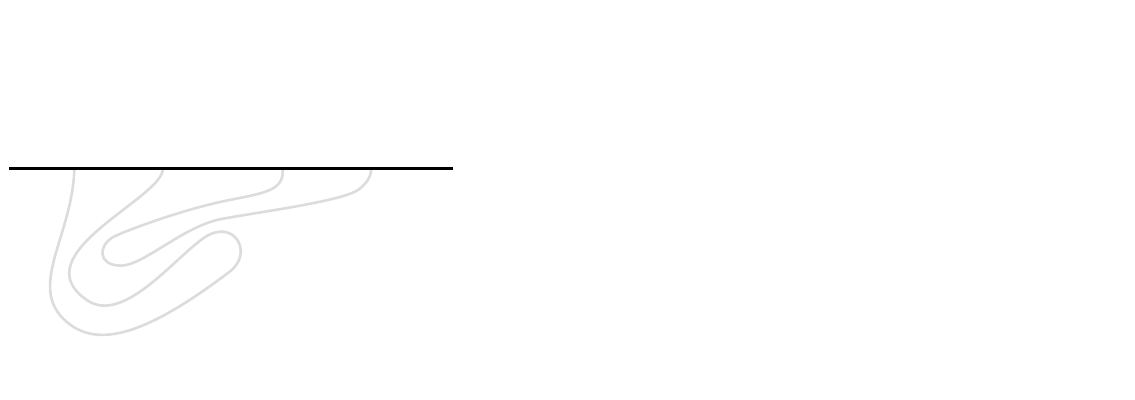
	\caption{A pointed open book $\cO$ (left), and the Heegaard diagram
    $\cH^\vee(\cO)$ of $(Y,\bL)$ (right). The intersection point $\xs(\cO)$ is shown.}
\label{fig:15}
\end{figure}

\subsection{Naturality of the transverse invariant}
\label{sec:naturality}

Naturality of the transverse invariants is slightly subtle. We note that Baldwin,
Vela-Vick, and V\'ertesi \cite{BVVVTransverse}*{Theorem~3.1}
defined the BRAID invariant only up to an automorphism of knot Floer homology, as
their transition maps change the embedding of the link, as well as
the number of basepoints. We now sketch a slightly stronger naturality result.

\begin{prop}\label{prop:naturality}
Suppose that $\bL$ is a multi-pointed, transverse link in the contact 3-manifold $(Y,\xi)$.
If $\cO$ and $\cO'$ are two pointed, embedded open books compatible with
$(Y,\bL,\xi)$, then
\[
\Psi_{\cH^\vee(\cO)\to \cH^{\vee}(\cO')}([\xs(\cO)])=[\xs(\cO')],
\]
where $\Psi_{\cH^\vee(\cO)\to \cH^\vee(\cO')}$ denotes the map from naturality~\cite{JTNaturality},
and $\xs(\cO)$ and $\xs(\cO')$ are the canonical intersection
points representing the transverse invariant.
\end{prop}

\begin{proof}
The key ingredients for the proof of Proposition~\ref{prop:naturality} that
differ from the standard proofs of the invariance of the contact class,
are as follows. By Proposition~\ref{prop:pointed-to-partial}, if $\cO$ is a
pointed open book compatible with $(Y,\bL,\xi)$, then the partial open book
$\cO^\circ$ described in Section~\ref{sec:partial-open-books} is adapted to
the contact structure $\xi_{\bL}^\pitchfork$ on the sutured manifold $Y(\bL)$.
Furthermore, it is not hard to see that any partial open book for
$(Y(\bL),\xi_{\bL}^{\pitchfork})$ may be used to compute the transverse
invariants (even the minus versions). Hence,
Proposition~\ref{prop:naturality} is proven by following the proof of the
naturality of the sutured contact class of Honda, Kazez, and Mati\'c~\cite{HKMSutured}
using the relative Giroux correspondence for partial open books,
applied to $(Y(\bL), \xi^{\pitchfork}_{\bL})$.
\end{proof}

\subsection{Transverse invariants and positive saddle cobordisms}

In this section, we show that the dual of the link cobordism map for a
positive saddle  preserves the transverse invariant. We need the following
definition:

\begin{define}
Suppose that $(W,\cS)\colon (Y,\bL_0)\to (Y,\bL_1)$ is a decorated link cobordism, where
$W = [a,b] \times Y$ and $\cS = (S,\cA)$. We say that $(W,\cS)$ is a \emph{standard
$\zs$-saddle cobordism} if the following hold:
\begin{enumerate}
\item The projection $\pi \colon [a,b] \times Y \to [a,b]$ restricts to a Morse function on $S$,
    with a single index 1 critical point.
\item The index 1 critical point of $\pi|_S$ occurs in the subregion $S_{\zs} \subset S\setminus \cA$.
\item The restriction $\pi|_{\cA}$ has no critical points.
\end{enumerate}
\end{define}

A \emph{standard $\ws$-saddle cobordism} is defined similarly.

\begin{lem}\label{lem:saddles}
Suppose that $(W,\cS) \colon (Y,\bL_0,\xi) \to (Y,\bL_1,\xi)$ is an ascending link
cobordism with $W = [a,b] \times Y \subset \Symp(Y,\xi)$, which is a standard
$\zs$-saddle cobordism, whose critical point is positive. Then
\[
\left(F_{W, \cS}^\circ\right)^\vee\left(\tra^\circ(\bL_1)\right) = \tra^\circ(\bL_0),
\]
for $\circ \in \{\wedge,-\}.$
\end{lem}

\begin{proof}
By Lemma~\ref{lem:ascending-index-1-half-twist}, if $\bL_0$ and $\bL_1$ are
transverse links in $(Y,\xi)$, and $\bL_1$ differs from $\bL_0$ by a positive saddle
attachment, then we may pick embedded, pointed open books
\[
\cO_0=(\page, h_0, \ve{p})\quad \text{and} \quad \cO_1=(\page,h_1,\ve{p})
\]
for $(Y,\bL_0,\xi)$ and $(Y,\bL_1,\xi)$, respectively, such that
$h_1=\tw_\g^+\circ h_0$ for some path $\g$ connecting two basepoints
$p_1$, $p_2 \in \ps$, where $\tw^+_\g$ denotes a positive half-twist along $\g$.

Pick a basis of arcs $\ve{a} = (a_1,\dots, a_n)$ of $\page $ such that there is
exactly one arc, $a_i$, such that $h(a_i)\cap \g$ is non-empty, and
furthermore,
\[
|h(a_i) \cap \g| = 1.
\]

Let $(\Sigma, \as, \bs, \ws, \zs)$ and $(\Sigma, \as, \bs', \ws, \zs)$ be the multi-pointed link diagrams
for $(Y, \bL_0)$ and $(Y, \bL_1)$ constructed from the pointed open books $\cO_0$ and $\cO_1$
and the arc basis $\ve{a}$, respectively, as described in Section~\ref{sec:construction-braid}.
Consider the triple diagram $(\Sigma,\as,\bs,\bs',\ws,\zs)$,
where $\as$, $\bs$, and $\bs'$ intersect on $\page \times \{\tfrac{1}{2}\}$
as shown on the right of Figure~\ref{fig:14}.

\begin{figure}[ht!]
\centering
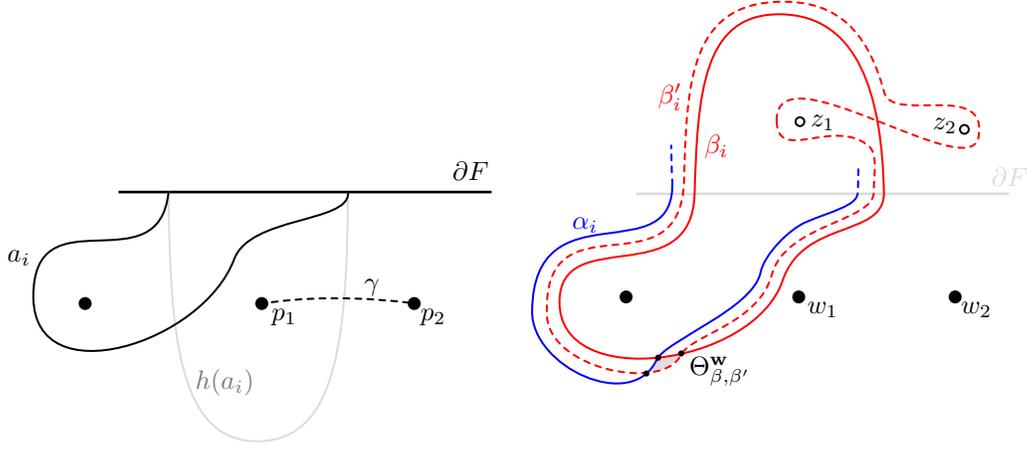
\caption{This model computation shows that the reverse of a positive saddle cobordism
preserves $\tra^\circ$. Note that the top half of the right-hand side represents
$\bar{\page} \times \{0\}$. Hence a positive half-twist on $\page$
appears as a negative one on $\bar{\page} \times \{0\}$.}
\label{fig:14}
\end{figure}

Let $B \subset Y$ denote the band for the link $L_0$, obtained by projecting
the saddle cobordism into $Y$. In terms of
\cite{ZemCFLTQFT}*{Definition~6.2}, the triple $(\Sigma,\as,
\bs,\bs',\ws,\zs)$ is \emph{subordinate} to the band $B$. Consequently, we
may use this triple to compute the link cobordism maps for the saddle.
Similarly, the triple $(\Sigma,\bs,\bs',\as,\ws,\zs)$ is subordinate to the
reversed band $B^\vee$, now viewed as being attached to $\bL_1$ to form $\bL_0$.

There are two canonical intersection points
\[
\Theta_{\b,\b'}^{\ws}, \Theta_{\b,\b'}^{\zs} \in \bT_{\b}\cap \bT_{\b'},
\]
distinguished by their multi-gradings. The element
$\Theta_{\b,\b'}^{\ws}$ is the top $\gr_{\ws}$-graded intersection point of
$\bT_{\b} \cap \bT_{\b'}$, while $\Theta_{\b,\b'}^{\zs}$ is the top
$\gr_{\zs}$-graded intersection point. The type-$\ws$ and type-$\zs$ saddle
maps are defined via the triangle counts
\begin{equation}\label{eq:saddles-def}
F_{B^\vee}^{\ws}(-) := F_{\b,\b',\a}(\Theta_{\b,\b'}^{\zs} \otimes -) \qquad \text{and} \qquad
F_{B^\vee}^{\zs}(-):= F_{\b,\b',\a}(\Theta_{\b,\b'}^{\ws} \otimes -).
\end{equation}
The map $F_{B^\vee}^{\ws}$ in equation~\eqref{eq:saddles-def} is the
cobordism map for a standard $\ws$-saddle cobordism, and the
$F_{B^\vee}^{\zs}$ is the map for the cobordism map for a standard
$\zs$-saddle cobordism. It is straightforward to see that
\[
F_{B^\vee}^{\zs}=\left(F_{B}^{\zs}\right)^\vee,
\]
as a special case of equation~\eqref{eq:duality-maps}.

For the main claim, it is sufficient to prove
\begin{equation}\label{eq:Fwt1=t0}
F^{\zs}_{B^\vee}(\tra^\circ(\cO_1))=\tra^\circ(\cO_0).
\end{equation}
Equation~\eqref{eq:Fwt1=t0} is verified via the model computation shown in
Figure~\ref{fig:14}, since there is a unique holomorphic triangle with
corners at $\Theta_{\b,\b'}^{\ws}$ and $\xs(\cO_1)$ that has
zero multiplicity at the $\ws$ basepoints. Furthermore, this triangle has
index 0, and has third corner at $\xs(\cO_0)$. The proof is complete.
\end{proof}

\subsection{Transverse invariants and positive birth cobordisms}

\begin{lem}\label{lem:birth}
Suppose that $\bL$ is a multi-pointed, transverse link in the contact 3-manifold $(Y,\xi)$,
and $\bU$ is a doubly pointed, transverse unknot in $Y$, which is unlinked from $\bL$.
Suppose that
\[
(W,\cS) \colon (Y,\bL,\xi) \to (Y,\bL \cup \bU, \xi)
\]
is an ascending link
cobordism in $W = [a,b] \times Y \subset \Symp(Y,\xi)$ such that $\cS$ is equal
to $L \times I \sqcup \cD_0$, where $\cD_0$ is a positive birth. Furthermore, assume
the dividing set on $\cD_0$ consists of a single arc and
is $I$-invariant on $L \times I$. Then
\[
\left( F^\circ_{W,\cS}\right)^\vee\left(\tra^\circ(\bL\cup \bU)\right)=\tra^\circ(\bL),
\]
for $\circ\in \{\wedge,-\}$.
\end{lem}
\begin{proof}

By Lemma~\ref{lem:open-books-index-0}, we may construct pointed open books $\cO$ and
$\cO'$ for $\bL$ and $\bL\cup \bU$, respectively, such that
$\cO = (\page,h,\ve{p})$ and $\cO' = (\page,h,\ve{p} \cup \{p'\})$, where $p'$ is a new
basepoint which is close to $\d \page$; see Figure~\ref{fig:20}.

There is a chain isomorphism
\[
\CFL^\circ(\cH^{\vee}(\cO')) \iso  \CFL^\circ(\cH^\vee(\cO)) \otimes \langle \theta^-, \theta^+\rangle,
\]
where $\theta^-$ and $\theta^+$ are distinguished by the Maslov grading. By construction,
\[
\tra^\circ(\cO')=\tra^\circ(\cO)\times \theta^-.
\]
The map $\left(F^\circ_{W,\cS}\right)^\vee$ coincides with the map
for a death cobordism, which is given by the formula
\[
(F^\circ_{W, \cS})^{\vee}(\xs\times \theta)=\begin{cases} \xs& \text{ if } \theta=\theta^-,\\
0& \text{ if } \theta=\theta^+,
\end{cases}
\]
extended equivariantly over the action of $v$, if $\circ=-$.  The result immediately follows.
\end{proof}

\begin{figure}[ht!]
\centering
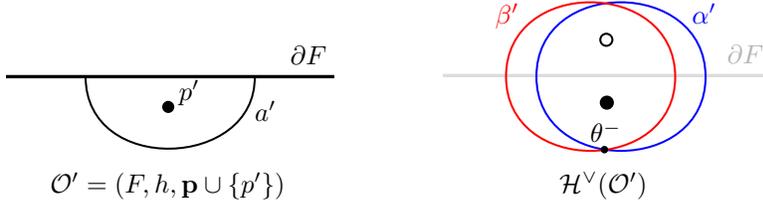
\caption{Left: Adding a new point $p'$ and an arc $a'$ to an open book near $\d \page $,
as in Lemma~\ref{lem:birth}.
Right: The resulting Heegaard diagram.}
\label{fig:20}
\end{figure}

\subsection{Transverse invariants and quasi-stabilizations}

In this section, we prove that the transverse invariant is functorial with
respect to certain simple decorations on cylindrical link cobordisms. Our
proof is a repackaging of Baldwin, Vela-Vick and V\'{e}rtesi's proof that the
BRAID invariant is invariant under positive Markov stabilizations
\cite{BVVVTransverse}*{p.~948}.

We recall the TQFT framework of the third author for adding and removing
basepoints from a link, as described in \cite{ZemQuasi} and
\cite{ZemCFLTQFT}. Suppose that $(L,\ws,\zs)$ is a multi-pointed link in $Y$,
and $(w,z)$ is a new pair of basepoints on $L$, which are in the same
component of $L\setminus (\ws\cup \zs)$. Suppose, for definiteness, that $w$
follows $z$ with respect to the orientation of $L$. The third author
described chain maps
\[
\begin{split}
S_{w,z}^+, T_{w,z}^+ &\colon \CFL^\circ(Y,L,\ws,\zs)\to \CFL^\circ(Y,L,\ws\cup \{w\}, \zs\cup \{z\}) \text{ and } \\
S_{w,z}^-, T_{w,z}^- &\colon \CFL^\circ(Y,L,\ws\cup \{w\}, \zs\cup \{z\})\to \CFL^\circ(Y,L,\ws,\zs).
\end{split}
\]
The maps $S_{w,z}^{\pm}$ and $T_{w,z}^{\pm}$ correspond to special dividing
sets on the cylinder $I\times L$, as shown in Figure~\ref{fig:8}. Note that
any dividing set on $I\times L$ can be obtained, up to isotopy, by composing
such dividing sets.

The maps $S_{w,z}^+$ and $T_{w,z}^+$ are obtained by stabilizing the
Heegaard diagram in a special way; see Figure~\ref{fig:9}. This operation is
called a \emph{linked index 0/3} stabilization in \cite{BVVVTransverse}, and
a \emph{quasi-stabilization} in \cite{MOIntegerSurgery}. The third author
showed that the operation is natural \cite{ZemQuasi}*{Theorem~A}, and gave a link
cobordism interpretation \cite{ZemCFLTQFT}*{Section~4} using the decorated cobordisms in
Figure~\ref{fig:8}.

\begin{figure}[ht!]
\centering
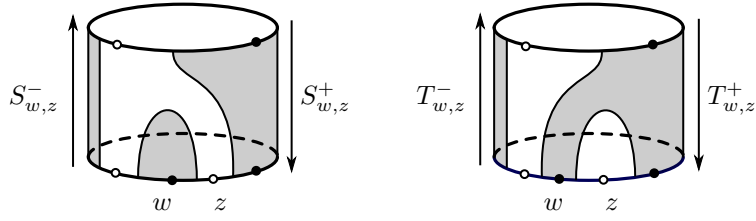
\caption{Dividing sets for the maps $T_{w,z}^+$, $T_{w,z}^-$, $S_{w,z}^+$, and $S_{w,z}^-$.}
\label{fig:8}
\end{figure}

\begin{lem}\label{lem:quasi-stabilization}
Suppose that $\bL=(L,\ws,\zs)$ is a multi-pointed, transverse link in $(Y,\xi)$, and $\bL'$ is obtained by adding two new, adjacent basepoints $w$ and $z$ to $\bL$. For $\circ \in \{\wedge,-\}$, we have
\[
\left(S_{w,z}^+\right)^\vee(\tra^\circ(Y,\bL',\xi)) = \tra^\circ(Y, \bL,\xi) \quad
 \text{and} \quad \left(T_{w,z}^-\right)^\vee(\tra^\circ(Y,\bL,\xi)) = \tra^\circ(Y,\bL',\xi).
\]
\end{lem}

\begin{proof}
Our proof follows by analyzing Baldwin, Vela-Vick, and V\'ertesi's proof of the
invariance of the BRAID invariant under positive Markov stabilizations; see
\cite{BVVVTransverse}*{Figure~11}. We pick a pointed open book $\cO=(\page ,h,\ve{p})$ for $(Y,\bL,\xi)$.
We may construct a pointed open book $\cO'=(\page ,h',\ve{p}')$ for $(Y,\bL',\xi)$, as follows.
 We define $\ve{p}'=\ve{p}\cup \{p\}$, where $p$ is a point near $\d \page $. We pick a path $\g$
from $p$ to another point $p'\in \ve{p}$, and set $h'$ to be
$\tw^+_\g \circ h$. We pick a basis of arcs $\ve{a}$ for $\cO$ such that $p$
and $p'$ lie in the same component of $\page  \setminus \ve{a}$, and pick a basis
$\ve{a}'$ for $\cO'$ by adjoining an arc $a$ that bounds a bigon with
$\d \page $ containing $p$.

\begin{figure}[ht!]
\centering
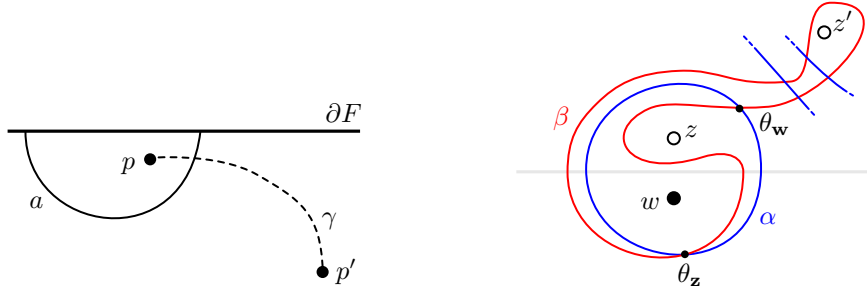
\caption{Left: The effect of a positive Markov stabilization on the open book $\cO$. Right:
Its effect on the Heegaard diagram $\cH^\vee(\cO)$.}
\label{fig:9}
\end{figure}

The Heegaard diagram $\cH(\cO')$ is shown in Figure~\ref{fig:9}.
There is an isomorphism of chain complexes
\[
\CFL^\circ(\cH^\vee(\cO'))\iso \CFL^\circ(\cH^\vee(\cO)) \otimes
\langle \theta_{\ws}, \theta_{\zs} \rangle,
\]
where $\theta_{\ws}$ and $\theta_{\zs}$ are the two intersection points
labeled in Figure~\ref{fig:9}. The point $\theta_{\ws}$ is the higher-graded
intersection point with respect to $\gr_{\ws}$, and
$\theta_{\zs}$ is the higher-graded intersection point with respect to $\gr_{\zs}$.

The maps $T_{w,z}^{\pm}$ and $S_{w,z}^{\pm}$ are defined via the formulas
\[
T_{w,z}^+(\xs) = \xs \times \theta_{\zs}\quad \text{and} \quad
T_{w,z}^-(\xs\times \theta) =
\begin{cases}
\xs &\text{if } \theta = \theta_{\ws}, \\
0& \text{if } \theta = \theta_{\zs},
\end{cases}
\]
and
\[
S_{w,z}^+(\xs) = \xs \times \theta_{\ws}\quad \text{and} \quad
S_{w,z}^-(\xs\times \theta) =
\begin{cases}
0 &\text{if } \theta = \theta_{\ws}, \\
\xs& \text{if } \theta = \theta_{\zs}.
\end{cases}
\]

Note that
\[
\tra^\circ(\cO')=\tra^\circ(\cO)\times \theta_{\zs}.
\]
Hence, we obtain that
\begin{equation}
S_{w,z}^-(\tra^\circ(Y,\bL',\xi))=\tra^\circ(Y,\bL,\xi)\quad \text{and} \quad T_{w,z}^+(\tra^\circ(Y,\bL,\xi))=\tra^\circ(Y,\bL',\xi).\label{eq:quasi-stab-without-dualizing}
\end{equation}
The main  result now follows from equation~\eqref{eq:quasi-stab-without-dualizing}
by noting that $S_{w,z}^+$ is dual to $S_{w,z}^-$, and $T_{w,z}^+$ is dual to $T_{w,z}^-$.
\end{proof}

\subsection{Proof of Theorem~\ref{thm:ascending-Stein}}
\label{sec:proofs}

In this section, we prove Theorem~\ref{thm:ascending-Stein}.
The following terminology will be convenient:

\begin{define}
Suppose that $\cS=(S,\cA)$ is a decorated surface, equipped with a Morse
function $\phi$ such that $\phi|_{\cA}$ is also Morse. We say that a critical
point $p$ of $\phi|_{\cA}$ is of \emph{type $S^+$} if $\phi|_{\cA}$ has a
local minimum at $p$, and the region $\cS_{\ws}$ lies immediately above $p$,
and $\cS_{\zs}$ lies immediately below, as in the reverse of the
left-hand side of Figure~\ref{fig:8}. A critical
point is similarly said to be of \emph{type} $S^-$, $T^+$, or $T^-$ if it
satisfies the analogous configuration shown in Figure~\ref{fig:8}.
\end{define}

We now reformulate the definition of $\ws$-arboreal decorations.

\begin{lem}\label{lem:ascending-restatement}
Suppose that that $\cS=(S,\cA)$ is a surface with divides, with $\d
S=-L_0\cup L_1$. Suppose $\phi\colon S\to [0,1]$ is a Morse function,
and $\phi$ has only index 1 and 2 critical points. Then $\cS$ is
$\ws$-arboreal with respect to $\phi$ if and only if the dividing set $\cA$
may be isotoped so that the following hold:
\begin{enumerate}
\item\label{alt-def:arboreal-1} All index 2 critical points of $\phi$ occur along $\cA$.
\item\label{alt-def:arboreal-2} All index 1 critical points of $\phi$ occur in $\cS_{\zs}$.
\item\label{alt-def:arboreal-3} The function $\phi|_{\cA}$ is Morse, with only
    type $S^-$ and $T^+$ critical points.
\end{enumerate}
\end{lem}

\begin{proof}
 Suppose first $\cS$ is $\ws$-arboreal with respect to $\phi$, and let $F$ be
 the forest in the definition. Perturbing $F$ slightly, we may assume that
 $F$ is disjoint from the critical set of $\phi$. Next, we add an edge to $F$ for each
 index 2 critical point, by flowing from each critical point downward along the
 gradient of $\phi$, until it hits either an edge of $F$ or $L_0$. If it hits
 $F$, we are done. If the downward trajectory hits $L_0$, we then isotope it
 near $L_0$ so that it intersects an edge of $F$. Let $\cA = \cl(N(F) \setminus \d S)$,
 for a suitably chosen regular neighborhood $N(F)$ of $F$. After moving
 $\cA$ very slightly near the index 2 critical points, so they lie on $\cA$,
 it is straightforward to check that it has the stated properties.

 Conversely, suppose that $\cS$ satisfies the properties
 \eqref{alt-def:arboreal-1}, \eqref{alt-def:arboreal-2},
 \eqref{alt-def:arboreal-3}. We may decompose $\cS$ into a composition of
 standard $\ws$-saddles, deaths occurring along the dividing set, as well as
 cylindrical link cobordisms where $\phi|_{\cA}$ has a single critical point,
 which further has the configuration of type $T^+$ or $S^-$.
 Since the $\ws$-arboreal property is preserved by
 composition, it is sufficient to check the $\ws$-arboreal condition for each
 of the above elementary cobordisms, which is straightforward.
\end{proof}

\begin{proof}[Proof of Theorem~\ref{thm:ascending-Stein}]
According to the work of Eliashberg~\cite{EliashbergStein} and Gompf~\cite{GompfStein}, the Morse
function $\phi$ has critical points only of index 0, 1, and 2. Passing an index~1
critical point of $\phi$ has the effect of removing two standard Darboux
balls, and gluing a standard $S^2 \times I$ along two convex copies of $S^2$.
The 2-handles are attached along contact $(-1)$-framed Legendrians.
Furthermore, if $W$ is connected and $Y_0$ is non-empty,
then we may assume that there are no 0-handles.

By Lemma~\ref{lem:ascending-surface}, $\phi|_S$ has only index 0 and 1
critical points. By Lemma~\ref{lem:ascending-restatement}, after an isotopy
of the decoration, we may assume that the dividing set intersects any index 0
critical point of $\phi|_S$, all index 1 critical points of $\phi|_S$ occur
in $\cS_{\zs}$, and also that every critical point of $\phi|_{\cA}$ is of
type $S^+$ or $T^-$. It is sufficient to prove the case when $(W,\cS)$ is an
ascending link cobordism, with $\ws$-anti-arboreal decoration, such that
$(W,\cS)$ contains at most 1 critical point of one of the functions $\phi$,
$\phi|_{\cS}$, or $\phi|_{\cA}$.

First, if $\phi$, $\phi|_{\cS}$, and $\phi|_{\cA}$ have no critical points,
and $\d W = -Y_0 \cup Y_1$, with $\phi(Y_i) = i$,
then there is a diffeomorphism between $W$ and $I \times Y$
that intertwines $\phi$ with the projection onto the first factor, and
intertwines the contact structure $i_V \omega$ on $\phi^{-1}(c)$ with some
fixed contact structure $\xi$ on $Y$, for all $c \in I$; see
\cite{HaydenSteinSurfaces}*{Lemma~4.12}. This diffeomorphism maps $S$ to the
trace of a transverse isotopy of links in  $(Y,\xi)$. A straightforward
Moser-type argument implies that $S$ is the trace of a contact isotopy of
$(Y,\xi)$. Furthermore, this isotopy may be assumed to be compatible with the
dividing set. Hence, functoriality of $\tra^\circ$ , for this cobordism,
follows from the naturality of $\tra^\circ$,
Proposition~\ref{prop:naturality}.

The proof of the statement for Weinstein 1-handles and 2-handles adapts
immediately from the proof given by Ozsv\'{a}th and Szab\'{o}
\cite{OSContactStructures} for functoriality of the contact invariant of
closed 3-manifolds with respect to Stein cobordisms, as in \cite{JCob}*{Theorem~11.24}.
See \cite{OSContactStructures}*{Theorem~4.2}  for the proof of the statement
about Weinstein 2-handles. For Weinstein 1-handles, the computation is as
follows. If $\cO=(S,P,h)$ is an open book for $(Y_0,\bL_0)$, and
$(Y_1,\bL_1)$ is obtained by adding a Weinstein 1-handle, then we may obtain
an open book for $(Y_1,\bL_1)$, by attaching a band to $\d S$, and extending
$h$ via the identity over the band. A straightforward computation verifies
the claim in this case.

It remains to verify the statement for cobordisms which have a single
critical point of $\phi|_S$ or a critical point of $\phi|_{\cA}$. The claim
for index 0 critical points of $\phi|_{S}$ is Lemma~\ref{lem:birth}. The
claim for index 1 critical points of $\phi|_{S}$ is Lemma~\ref{lem:saddles}.
Finally, the claim for critical points of $\phi|_{\cA}$ is
Lemma~\ref{lem:quasi-stabilization}. The proof is complete.
\end{proof}

We now sketch the proof of the perturbed version:

\begin{proof}[Proof of Proposition~\ref{prop:ascending-perturbed-coefficients}]
The proof is similar to the proof of Theorem~\ref{thm:ascending-Stein}. We
decompose the cobordism $(W,\cS)$ into a sequence of positive births,
positive saddles, $S^+$ and $T^-$ decorations, and Weinstein 1- and
2-handles. In between the moves, we may also have to perform open book
isotopies, and Giroux stabilizations of the open books.

The key observation is that, for each of the associated maps, if $\cO_0$ is an open book for the
incoming end, and $\cO_1$ is an open book for the outgoing end, then the dual
of the cobordism or naturality map sends the canonical intersection point
$\xs(\cO_1)\otimes 1$ to $\xs(\cO_0)\otimes e^{(a_1,\dots, a_n)}$. For
example, the dual of the cobordism map for a positive saddle counts exactly
one homotopy class of triangles when applied to $\xs(\cO_1)$, as illustrated
in
Figure~\ref{fig:14}. The other types of elementary ascending and Weinstein cobordisms are analyzed similarly.
\end{proof}

\section{Proof of Theorem~\ref{thm:main}}

We are now ready to prove Theorem~\ref{thm:main} from the introduction.

\begin{proof}[Proof of Theorem~\ref{thm:main}]
According to Rudolph~\cite[Lemma~2]{RudolphSlice}, the untwisted positive Whitehead
double of a nontrivial strongly quasipositive knot is strongly quasipositive.
Furthermore, the genus one Seifert surface obtained by taking an untwisted annulus
about the knot, and plumbing it with a $+1$ twisted annulus about the unknot is
a quasipositive Seifert surface for the Whitehead double.

Let $K_0$ be a non-trivial strongly quasipositive knot in $S^3$,
write $K_1$ for its untwisted positive Whitehead double,
and let $K$ be the untwisted positive Whitehead double of $K_1$.
Then both $K_1$ and $K$ are strongly quasipositive.
Furthermore, they are topologically slice by the work of Freedman~\cite{FQ},
as they have Alexander polynomial one. Let $S$ be the standard genus one
Seifert surface of $K$, obtained by plumbing an untwisted annulus about $K_1$
with a $+1$ twisted annulus about the unknot. In particular, $\gamma := K_1$
lies on the surface $S$, and the surface framing of $\gamma$ is trivial.
By construction, $\gamma$ is topologically slice, so it bounds a topological
disk in the complement of $S$. Furthermore, $\gamma$ is homologically nontrivial in $S$.

The pair $(S, \gamma)$ satisfies the assumptions of Theorem~\ref{thm:topiso}.
For every $n \in \Z$, choose a knot $J_n$ such that its
Alexander polynomial $\Delta_{J_n}(z)$ has $n$ irreducible factors.
For example, suppose that $J$ is a nontrivial knot such that $\Delta_{J}(z)$ is irreducible,
and take $J_n := \#_n J$.
Let $C_n := (J_n)_{\tw}$ be the 1-twist self-concordance of $J_n$,
as in Definition~\ref{def:1-twist}. As the automorphism $\phi$ of $S^3$ used
in the construction of $C_n$ is isotopic to $\id_{S^3}$ through automorphisms of $S^3$
that fix $J_n$ as a based knot pointwise throughout, $\phi$ induces the identity
on $\HFKh(J_n)$, and so $\Lef_z(C_n) = \Delta_{J_n}(z)$.
Finally, $F_{B^4, S} \neq 0$ by Corollary~\ref{cor:BRAID}, as $S$ is strongly quasipositive.
It follows from Theorem~\ref{thm:not-diffeomorphic} that the 1-twist rim surgered
surfaces $S(\gamma, C_n)$ are pairwise non-diffeomorphic.
\end{proof}

\bibliographystyle{plain}
\bibliography{biblio}
\end{document}